\documentclass[11pt]{amsart}
\usepackage[utf8]{inputenc}
\usepackage[margin=1in]{geometry}

\usepackage{caption}
\usepackage{subcaption}

\usepackage{lmodern}
\usepackage{mathtools}
\usepackage{xcolor}
\usepackage{amsthm}
\usepackage{amsmath}
\usepackage{comment}
\usepackage{ stmaryrd }
\usepackage{appendix}

\usepackage{amsfonts}
\usepackage{amssymb}

\usepackage{hyperref}

\usepackage{tikz}
\usepackage{tikz-3dplot-circleofsphere}

\allowdisplaybreaks

\newcommand{\R}{\mathbb{R}}

\newcommand{\Sb}{\mathbb{S}}
\newcommand{\B}{\mathcal{B}}

\newcommand{\M}{\mathcal{M}}

\newcommand{\LL}{\mathcal{L}}

\newcommand{\dd}{\mathrm{d}}
\newcommand{\ddd}{\mathrm{d}}

\newtheorem{theorem}{Theorem}
\newtheorem{proposition}{Proposition}
\newtheorem{definition}{Definition}

\newtheorem{lemma}[proposition]{Lemma}
\newtheorem{remark}{Remark}
\newtheorem{corollary}[proposition]{Corollary}

\numberwithin{equation}{section}

\title[Compactness of the linearized Boltzmann operator with resonant collisions]{Compactness property \\of the linearized Boltzmann operator \\for a polyatomic gas undergoing resonant collisions}
 \thanks{This work has been partially supported by the COST Action CA18232 {\it Mathematical models for interacting dynamics on networks} and by the PHC Pessoa proejct n. 47871W}
\author[T. Borsoni]{Thomas Borsoni}
\address{T.B., L.B.: Sorbonne Université, CNRS, Université Paris Cité, Laboratoire Jacques-Louis Lions (LJLL), F-75005 Paris, France}
\email{thomas.borsoni@sorbonne-universite.fr}
\email{laurent.boudin@sorbonne-universite.fr}

\author[L. Boudin]{Laurent Boudin}

\author[F. Salvarani]{Francesco Salvarani}
\address{F.S.: L\'eonard de Vinci P\^ole Universitaire, Research Center\\
92916 Paris La D\'efense, France \& 
Dipartimento di Matematica ``F. Casorati''\\
Universit\`a degli Studi di Pavia\\
Via Ferrata 1, 27100 Pavia, Italy}

\email{francesco.salvarani@unipv.it}

\date{\today}

\begin{document}

\maketitle

\begin{abstract}
In this paper, we investigate a compactness property of the linearized Boltzmann operator in the context of a polyatomic gas whose molecules undergo resonant collisions. The peculiar structure of resonant collision rules allows to tensorize the problem into a velocity-related one, neighbouring the monatomic case, and an internal energy-related one. Our analysis is based on a specific treatment of the internal energy contributions. We also propose a geometric variant of Grad's proof of the same compactness property in the monatomic case.
\end{abstract}

\bigskip


\section{Introduction}

The study of linearized collisional operators in kinetic theory has been addressed in the literature long after the introduction of the corresponding quadratic operators \cite{MR0255203}. In fact, taking as an initial reference point Boltzmann's article \cite{bol}, the first questions about the linearized Boltzmann operator were studied several years later by Hilbert in \cite{MR1511713}. Hilbert wrote the linearized operator in an integral form, in the case of a hard-sphere cross section in three spatial dimensions, and proved its compactness. Subsequently, Hecke \cite{MR1544518} and Carleman \cite{MR0098477} extended Hilbert's results and proved that the linearized Boltzmann operator is of Hilbert-Schmidt type.

A substantial improvement in the understanding of those compactness properties was given in \cite{grad2} by Grad who, using an ingenious geometrical argument, proved the compactness of the linearized operator for a wide range of cross sections satisfying the so-called angular cut-off condition, which he previously introduced in \cite{grad1}. He adapted Hecke's proof by stating that the linearized Boltzmann operator was still of Hilbert-Schmidt type in a weighted $L^2$ space.

For a long time, no further significant steps related to the problem were proposed. However, many natural questions remained open, such as, for example, the extension of Grad's result to linearized Boltzmann operators without imposing the angular cut-off condition, or the study of compactness properties for gaseous mixtures, described by systems of coupled Boltzmann equations. Regarding the weakening of assumptions about the cross section, we highlight that an important result of compactness, which does not suppose any small deflection cut-off assumption on the cross sections, is due to Mouhot and Strain \cite{MR2322149}.

When dealing with gaseous mixtures, the Grad geometric argument may fail. Indeed, if the species involved in the mixture have different molar masses, the symmetry between pre- and post-collisional velocities is lost. This symmetry was crucial in Grad's strategy. A new argument to recover the compactness property for gaseous monatomic mixtures with different masses has been proposed in \cite{MR3005625}, where the mono- and multi-species cases are treated in two different ways: Grad's method still applies in the case of mono-species collisions or when the masses of the colliding particles are equal, whereas a pre/post-collisional velocity-crossing argument is used to deal with collisions between particles with different masses. Note that, however, it degenerates when the masses of the colliding particles are equal. We refer to \cite{MR3557717} for a more complete state-of-the-art on the compactness of the linearized Boltzmann operator and on the linearized operator for mixtures of monatomic gases.

The internal structure of polyatomic molecules influences the energy transfer during collisions, so that variables describing the internal state of a given molecule, such as its angular velocity or vibration mode, must be taken into account in the model, and appear in the energy conservation law during a collision. The corresponding kinetic models are based on various different approaches. The first one consists in introducing a set of discrete internal energy variables and consider a set of particle densities satisfying a system of coupled kinetic equations. The compactness of the linearized collision operator in that case has been studied very recently in \cite{https://doi.org/10.48550/arxiv.2201.01365}.

Another approach is based on the introduction of a continuous internal energy variable $I \in \R_+$ and the model describes the evolution of the particle density in an extended phase space, the variables then being time, position, velocity and internal energy. A mechanistic procedure for binary encounters in polyatomic gases was provided by Borgnakke and Larsen \cite{bor-lar-75}, and the corresponding linearized collisional operator was recently studied in \cite{https://doi.org/10.48550/arxiv.2201.01377,BrullShahineThieullen}.

In both approaches, one considers the molecule internal energy as the variable encapsulating its internal structure. A generalisation was proposed in \cite{borsoni2022general}, where the internal variable is rather the internal state of the molecule, for instance its angular velocity and vibrational modes. It was proven that this consideration of the states is equivalent to a model considering directly a continuous internal energy variable $I$ in the set $\R_+$ endowed with some measure $\mu$, called energy law, computable from the model of the molecule and fully characterising it. The case when $\mu$ is a discrete measure corresponds to the model with discrete energy levels, and the case when $\mu$ has a density with respect to the Lebesgue measure corresponds to the model of Borgnakke-Larsen models \cite{bor-lar-75,MR1277241,MR1611828,MR2118066}.

In this paper, we place ourselves in this latter setting, considering the internal energy $I$ in $\R_+$ endowed with the measure $\mu$. The underlying model on the molecule is assumed to be semi-classical, in the sense that rotation should be described using classical mechanics, and vibration using quantum mechanics. Recalling that $\mu$ is explicitly computable from the model of the molecule, the semi-classical setting allows us to ensure that $\mu$ has a density with respect to the Lebesgue measure, which can upper- and lower-bounded by power laws around 0 and infinity, which is helpful in the proof of our main theorem.

The specific and singular case we study is when the molecules of the gas undergo resonant collisions. That means that the microscopic internal and kinetic energies are separately conserved and cannot be handled by using the Borgnakke-Larsen procedure. This behaviour has been observed in some physical situations, such as in the collisions between selectively excited CO$_2$ molecules \cite{lom-fag-pac-gro-15}.

A kinetic operator describing resonant collisions has recently been proposed in \cite{boudin:hal-03629556}, but the study of the compactness of the corresponding linearized operator has not yet been carried out. The purpose of this article is hence to fill in this gap and therefore to study the compactness property of the associated linearized operator.

Although our approach to the problem is consistent with the traditional strategy and consists in the rewriting of the linearized operator in a kernel form, its practical realization required to face specific difficulties and allowed to introduce a new viewpoint, mainly in the treatment of the internal energy variables. In the process, we also propose a variant change of variables, alternative to the one proposed by Grad in \cite{grad2}, which may be helpful later, in the polyatomic mixture case.

Besides, it is clear that the properties of the cross section may heavily influence the proof strategies. However, the precise form of the cross section for resonant collisions cannot be completely deduced from purely mechanical arguments and the corresponding conservation laws. Several forms of cross-sections have been proposed in the literature for the Borgnakke-Larsen model, in both elastic and inelastic cases, see, for example, \cite{MR1618164,cercignani2000applications}. Nevertheless, since, up to our knowledge, the cross sections for resonant collisions have not been explicitly derived yet, we therefore aim to be as general as possible, by assuming that the cross section can be upper-bounded by a product of terms depending, separately, on the velocity and the internal energy variables. It allows to deal on the one hand with the velocity terms, and on the other hand, with the internal energy ones.

The article is organized as follows. In Section~\ref{S:model}, we describe the details of the model introduced in \cite{boudin:hal-03629556}, the hypotheses on the measure $\mu$ characterizing the chemical formula of the gaseous molecules, and the assumptions on the cross sections made for our analysis. Next, in Section~\ref{S:main}, we state our main compactness result, which is proven in Section~\ref{S:proof}. The non-multiplicative part of the linearized operator is split into several contributions, whose compactness are proven separately. Eventually, in the appendix, we propose our alternative to Grad's change of variables and provide its related geometric interpretation.

\section{Resonant collision model} \label{S:model}

We consider a gas composed of one species of polyatomic molecules with mass $m>0$. These molecules evolve in the three-dimensional Euclidean space $\R^3$, and we study their distribution function $f$, which depends on time $t \in \R_+$, position in space $x \in \R^3$, molecular velocity $v \in \R^3$ and internal energy $I \in \R_+$.

\subsection*{Molecule model}
In regards with the general framework defined in \cite{borsoni2022general}, and where internal states are considered, any state-based model can be reduced into a simpler energy-based one, where the internal variable is simply the internal energy, through reduction process. To be more accurate, we work in a semi-classical setting: we choose to describe the rotation of the molecule by using the standard approach of classical mechanics and the vibration part within the quantum mechanics framework, and we adopt an energy-based viewpoint. This means that we consider a measure $\mu$ on the set of internal energies $\R_+$, called the \emph{energy law}, which fully encapsulates the internal structure of the molecule, and which is absolutely continuous with respect to the Lebesgue measure. In the semi-classical case, the density of $\mu$ is typically a scale function, or a sum of pieces of square-roots, which implies that it behaves like a power law around 0, typically a constant or the square-root function, and behaves (or tends to behave) like another power law around infinity. 

Hence, the previous considerations justify our choice about the following properties of $\mu$, which will be supposed from now on. 
\begin{definition} \label{def:adm}
The measure $\mu$ is said admissible if
there exist $\beta_1$, $\beta_2 \geq 0$ and $C$, $C'>0$ such that
\begin{equation} \label{eqassumpmu1}
C \, I^{\beta_1} \leq \frac{\dd \mu (I)}{\dd I} \leq C' \, I^{\beta_1}, \qquad \text{for a.e. } I \in [0,1],
\end{equation}
and, for any $a > 0$, there exists $C_a > 0$ such that
\begin{equation} \label{eqassumpmu2}
C_a \, I^{\beta_2 - a} \leq \frac{\dd \mu (I)}{\dd I} \leq C' \, I^{\beta_2}, \qquad \text{for a.e. } I \geq 1.
\end{equation}
\end{definition}

Note that Definition~\ref{def:adm} is coherent with the hypotheses of the continuous-energy model proposed by Borgnakke and Larsen \cite{bor-lar-75}, see also \cite{MR2118066, MR1611828}. Moreover, the standard assumption $\dd \mu (I)= C \, I^{\alpha}\,\dd I$ with $\alpha \geq 0$, which allows to recover a temperature-independent heat capacity at constant volume and number of internal degrees of freedom, is of course covered by \eqref{eqassumpmu1}--\eqref{eqassumpmu2}.

We moreover define the internal partition function associated to the molecule model, for any temperature $T > 0$, by
\begin{equation} \label{eq:partition}
q(T) = \int_{\R_+} \exp\left(-\frac{I}{k_B T}\right)  \dd \mu(I),
\end{equation}
where $k_B$ stands for the Boltzmann constant. Assumptions \eqref{eqassumpmu1}--\eqref{eqassumpmu2} on $\mu$ straightforwardly imply that $q$ is well-defined and $\mathcal{C}^{\infty}$ on $\R_+^*$.

In order to describe the model, we first consider the microscopic dynamics, then we introduce our hypotheses on the cross section and, finally, we write the kinetic model which governs the time evolution of the density $f$.

\subsection*{Microscopic dynamics}

At fixed time and position, a given molecule is described by its velocity $v$ and its internal energy $I$. The couple $(v,I)$ is called the \textit{state} of the molecule. 

Let us consider the collision between two molecules. The pre-collision states are denoted by $(v,I)$ and $(v_*,I_*)$, and the post-collision ones by $(v',I')$ and $(v'_*,I'_*)$. The system of molecules being isolated, the total momentum and energy of the system are conserved during the collision, {\it i.e.}
\begin{align} 
    v + v_* &= v' + v'_*, \label{eq:conservationmoment}\\
    \frac12 m |v|^2 + I + \frac12 m |v_*|^2 + I_* &= \frac12 m |v'|^2 + I' + \frac12 m |v'_*|^2 + I'_*.\label{eq:conservationenergy}
\end{align}
Following \cite{boudin:hal-03629556}, we suppose that the collisions are \emph{resonant}, which means that \eqref{eq:conservationenergy} is a consequence of the two separate conservation laws
\begin{align}
    |v|^2 + |v_*|^2 &=  |v'|^2 + |v'_*|^2, \label{eq:conservationresonant1}\\
    I + I_* &= I' + I'_*,\label{eq:conservationresonant2}
\end{align}
ensuring that both kinetic and internal energy are separately conserved. In the resonant case, collision rules involving the velocities are the same as in the monatomic case. Consequently, Equations~\eqref{eq:conservationmoment} and \eqref{eq:conservationresonant1} imply the existence of $\sigma \in \Sb^2$ such that
\begin{equation} \label{eq:primevelocities}
    v' = \frac{v+v_*}{2} + \frac{|v-v_*|}{2} \, \sigma, \qquad v'_* = \frac{v+v_*}{2} - \frac{|v-v_*|}{2} \, \sigma.
\end{equation}
Besides, \eqref{eq:conservationresonant2} straightforwardly implies that
\begin{equation} \label{eqlawenergy}
    I'\in [0,I+I_*], \qquad \qquad I'_* = I + I_* - I'.
\end{equation}

We here choose to adopt the viewpoint of the general framework defined in \cite{borsoni2022general}, and do not parametrize the internal energy part. Indeed, if $I'$ is kept as an integration variable, the forthcoming computations may be simplified. We highlight that this choice is completely equivalent to the one of parametrizing the allocation of the internal energy towards $I'$ and $I'_*$, which is made in \cite{boudin:hal-03629556}. Hence, the results of the latter paper can be applied here.

\subsection*{Assumptions on the cross section}

We consider, in what follows, a cross section $B$ depending on $(v,v_*,I,I_*,I',\sigma) \in \R^3 \times \R^3 \times \R_+ \times \R_+ \times \R_+ \times \Sb^2$. Note that, because of \eqref{eqlawenergy}, the dependence with respect to $I'_*$ would be redundant. We emphasize that the cross section $B$ is positive if and only if a collision is possible. Hence, \eqref{eqlawenergy} implies that, for almost every $(v,v_*,I,I_*,I',\sigma) \in \R^3 \times \R^3 \times \R_+ \times \R_+ \times \R_+ \times \Sb^2$,
\begin{equation} \label{positivitykernel}
B(v,v_*,I,I_*,I',\sigma) > 0 \qquad \iff \qquad I' \leq I + I_*.
\end{equation}
Moreover, $B$ should satisfy the usual \emph{symmetry} and \emph{micro-reversibility} conditions: for almost every  $v$, $v_*$, $I$, $I_*$, $I'$ and $\sigma$,
\begin{align} \label{symmetrymicrorevkernel}
B(v_*,v,I_*,I,I+I_*-I',\sigma) &= B(v,v_*,I,I_*,I',\sigma), \\ 
B \left(\frac{v+v_*}{2} + \frac{|v-v_*|}{2} \, \sigma,\frac{v+v_*}{2} - \frac{|v-v_*|}{2} \, \sigma,I',I+I_*-I',I,\sigma \right) &= B(v,v_*,I,I_*,I',\sigma).
\end{align}
In order to obtain our main result, we moreover need $B$ to satisfy several assumptions:
\begin{itemize}
\item there exist functions $B_0$, $b_k$ and $b_i$ such that, for almost every $v$, $v_*$, $I$, $I_*$, $I'$ and $\sigma$,
\begin{multline} \label{eq:B0majorbkbi}
B(v,v_*,I,I_*,I',\sigma) = B_0 \left(|v-v_*|, \left|\cos \left(\widehat{v-v_*, \sigma} \right) \right|,I,I_*,I' \right) \\
\leq b_k\left(|v-v_*|, \left|\cos \left(\widehat{v-v_*, \sigma} \right) \right|\right) \, b_i(I,I_*) \, \mathbf{1}_{I' \leq I+I_*};
\end{multline}
\item there exist $\delta_1 \in [0,1)$, $\delta_2 \in [0,\frac12)$ and $C>0$ such that, for almost every $\rho>0$ and $\theta\in [0,2\pi]$,
\begin{equation} \label{eq:bkmajor}
b_k(\rho, |\cos \theta |) \leq C \left( |\sin \theta | \left( \rho^2 + \rho^{-1} \right) + \rho + \rho^{-\delta_1} + |\sin \theta |^{-\delta_2} \right);
\end{equation}
\item there exist $\gamma \in [0,2)$ and $C>0$ such that, for almost every $I$, $I_*>0$, 
\begin{equation} \label{eq:bimajor}
b_i(I,I_*) \leq C \, \frac{(I+I_*)^{\gamma / 2}}{\mu[0,I+I_*]}.
\end{equation}
\end{itemize}
We also define the averaged cross sections $\Bar{B}$ and $\Bar{B}_0$ by
\begin{align*}
\Bar{B}(v,v_*,I,I_*) &= \int_{\Sb^2} \int_{\R_+} B(v,v_*,I,I_*,I',\sigma) \, \dd \mu(I') \, \dd \sigma, \qquad \text{for a.e. } (v,v_*,I,I_*),\\
\Bar{B}_0(\rho,|\cos \theta |,I,I_*) &= \int_{\R_+} B_0(\rho,|\cos \theta |,I,I_*,I') \, \dd \mu(I'), \qquad \text{for a.e. } (\rho,\theta,I,I_*).
\end{align*}
Assumptions \eqref{eq:B0majorbkbi}--\eqref{eq:bimajor} on $B$ imply that, almost everywhere,
\begin{align}
 \Bar{B}(v,v_*,I,I_*) &\leq C \left( |v-v_*|^2 + |v-v_*|^{-1} \right) \, (I + I_*)^{\gamma / 2}, \label{eq:averagekernelbound}\\
\Bar{B}_0(\rho,|\cos \theta |,I,I_*) &\leq C \left( |\sin \theta | \left( \rho^2 + \rho^{-1} \right) + \rho + \rho^{-\delta_1} + |\sin \theta |^{-\delta_2} \right) \, (I + I_*)^{\gamma / 2}. \label{eq:averagekernelbound0}
\end{align}

\smallskip

\noindent The reader may notice that, by linearity, the result proven in this paper also holds when the cross section $B$ is a linear combination of cross sections following assumptions \eqref{eq:B0majorbkbi}--\eqref{eq:bimajor}. Since the case $\gamma=0$ is covered by Assumption \eqref{eq:bimajor}, it holds in particular when
$$
b_i(I,I_*) \leq  C \, \frac{1 + (I+I_*)^{\gamma / 2}}{\mu[0,I+I_*]}.
$$
Possible cross sections in the polyatomic case have been discussed in \cite{gamba2020cauchy} 

for the Borgnakke-Larsen collisions. Since, for all $\gamma \in (0,2)$,
$$
|v-v_*|^{\gamma} + (I + I_*)^{\gamma / 2} \leq C \left( |v-v_*|^2 + |v-v_*|^{-1} \right) \, \left[ 1 + (I + I_*)^{\gamma / 2} \right],
$$
we do cover here the general form of the cross section proposed in \cite[Equation (3.1)]{gamba2020cauchy}, with the exception of the critical case $\gamma=2$. Finally, the reader shall notice that the presence of $\mu[0,I+I_*]$ comes from the choice of not parametrizing the internal energies. Otherwise, with a parameter $r \in [0,1]$ such that $I' = r(I+I_*)$, then \eqref{eq:bimajor} would become $b_i(I,I_*) \leq C \, (I+I_*)^{\gamma / 2}$.

\subsection*{Kinetic operator, $H$ theorem}

We can now define the Boltzmann collision operator associated to the resonant model. For any measurable function $f$ such that it makes sense, it writes, for almost every $(v,I) \in \R^3 \times \R_+$,
\begin{multline*}
    Q(f,f)(v,I) = \int_{\R^3 \times (\R_+)^2 \times \Sb^2} \big(f(v',I') f(v'_*,I'_*) - f(v,I) f(v_*,I_*) \big) \\\!\!\!\phantom{\int} B(v,v_*,I,I_*,I',\sigma)  \, \dd v_* \, \dd \mu^{\otimes 2}(I_*,I')  \, \dd \sigma,
\end{multline*}
where $v',v'_*$ and $I'_*$ in the integral are defined in \eqref{eq:primevelocities}--\eqref{eqlawenergy}.

The $H$ theorem proved in \cite{boudin:hal-03629556} guarantees the existence of equilibrium states. In particular, the following result holds:
\begin{proposition} \label{prop:htheorem}
For all positive function $g = g(v,I)$ such that the following quantity is defined, we have
$$
\iint_{\R^3 \times \R_+} Q(g,g) (v,I) \, \log g(v,I) \, \ddd v \, \ddd \mu (I) \leq 0.
$$
Moreover, the three following properties are equivalent:
\begin{align}
    (i) \; &Q(\M,\M) = 0, \nonumber \\
    (ii) \; &\iint_{\R^3 \times \R_+} Q(\M,\M) (v,I) \, \emph{log}(\M)(v,I) \, \ddd v \, \ddd \mu (I) = 0,  \nonumber \\
    (iii) \; &\text{There exist } n \geq 0, \, u \in \R^3, \text{ and } T_k,\,T_i > 0 \text{ such that, for almost every } (v,I)\in \R^3 \times \R_+,  \nonumber \\
    &\hspace{20pt} \M(v,I) = \frac{n}{q(T_i)} \left( \frac{m}{2 \pi k_B T_k}\right)^{3/2} \, \exp \left(- \frac{m|v-u|^2}{2 k_B T_k}  - \frac{I}{k_B T_i} \right), \label{eqdefMaxwellian}
\end{align}
where we recall that $q$ is the partition function defined by \eqref{eq:partition}.
\end{proposition}
The equilibrium state given by \eqref{eqdefMaxwellian} writes as a product of two Gibbs distributions (which we also call Maxwellian functions by analogy to the monatomic case) with different temperatures $T_k$ and $T_i$. This is a consequence of the specific conservation laws of resonant collisions, which allow no exchange between the kinetic and internal parts.

\section{Main result} \label{S:main}

The standard perturbation setting for the Boltzmann equation combined with Proposition \ref{prop:htheorem} allows us to consider, for some Maxwellian $\M$ of the form \eqref{eqdefMaxwellian}, to define the linearized Boltzmann operator $\LL$ in the following way
\begin{equation} \label{eqdef:linearizedBoltzmann}
\LL = K - \Bar{\nu} \, \text{Id},
\end{equation}
where the collision frequency $\bar{\nu}$, appearing in the multiplicative operator, is given by
$$
\Bar{\nu}(v,I) = \int_{\R^3 \times \R_+}   \,  \M(v_*,I_*) \, \bar{B}(v,v_*,I,I_*) \,  \dd v_*  \, \dd \mu(I_*),
$$
and $K$ reads
\begin{align*}
    K g (v,I) =  \M(v,I)^{1/2} \, \int_{\R^3 \times (\R_+)^2 \times \Sb^2} \Big([\M^{-1/2}g](v',I') +  [\M^{-1/2} g](v'_*,I'_*) -  [\M^{-1/2} g](v_*,I_*) \Big)  \\
\times \M(v_*,I_*) \, B(v,v_*,I,I_*,I',\sigma) \, \dd v_* \, \dd \mu^{\otimes 2}(I_*,I')\, \dd \sigma.
\end{align*}
Let us now state our main result.
\begin{theorem} \label{theorem}
The operator $K$ is compact from $ \displaystyle L^2 \left( \R^3 \times \R_+, \; \ddd v \, \ddd \mu(I) \right)$ to itself.
\end{theorem}

Whitout loss of generality, we can suppose that the Maxwellian $\M$ have zero mean velocity, so that there exist $T_k$, $T_i > 0$ and $c > 0$ such that, for $(v,I) \in \R^3 \times \R_+$,
$$
\M(v,I) = c \, e^{-\frac{m |v|^2}{2 k_B T_k}} \, e^{-\frac{I}{k_B T_i}}.
$$
Finally, we define $M(v) =  c \, e^{-\frac{m |v|^2}{2 k_B T_k}}$ the velocity part of $\M$, so that we can write $\M(v,I) = M(v) \, e^{-\frac{I}{k_B T_i}}$. 

For the sake of simplicity, we assume $m=1$ in the following, without loss of generality. Indeed, $m$ only appears in $M$ through the constant $\frac{m}{2 k_B T_k}$, so that proving our result with $m=1$ for all $T_k>0$ is equivalent with proving it for all $m$, $T_k>0$.

\medskip

The main idea of the proof of Theorem \ref{theorem} is to use the specific structure of the resonant collision case in order to separate the variables $v$ and $I$, and to be able to use the known results on the monatomic case. 

\section{Proof of the main result} \label{S:proof}

We decompose $K$ into the sum of three operators, $K = K_1 + K_2 + K_3$, where
\begin{align*}
    K_1g (v,I)&\!= - \int_{\R^3 \times  \R_+}    \,  g(v_*,I_*) \,  \M(v,I)^{1/2} \M(v_*,I_*)^{1/2} \, \bar{B}(v,v_*,I,I_*) \,  \dd v_*  \, \dd \mu(I_*),\\
    K_2g (v,I)&\!= \!\!\int_{\R^3 \times  (\R_+)^2 \times \Sb^2}   \, [\M^{-1/2}g](v',I') \M(v,I)^{1/2} \M(v_*,I_*) \, B(v,v_*,I,I_*,I',\sigma)  \, \dd v_*  \dd \mu^{\otimes 2}(I_*,I') \, \dd \sigma, \\
    K_3 g (v,I) &\!= \!\!\int_{\R^3 \times  (\R_+)^2 \times \Sb^2}   [\M^{-1/2}g](v'_*,I'_*) \M(v,I)^{1/2} \M(v_*,I_*) \, B(v,v_*,I,I_*,I',\sigma)  \, \dd v_*   \dd \mu^{\otimes 2}(I_*,I') \, \dd \sigma.
\end{align*}
Applying the change of variable $\sigma \mapsto -\sigma$ and $I' \mapsto I+I_*-I'$ in $K_3$ allows to recover the same form as $K_2$, so that it is enough to prove that $K_1$ and $K_2$ are compact.

\subsection{Compactness of $K_1$}
In order to prove the compactness of $K_1$, we observe that it is in fact a Hilbert-Schmidt operator. We define the kernel, for almost every $v$, $I$, $v_*$, $I_*$,
$$
\kappa_1(v,I,v_*,I_*) = - \M(v,I)^{1/2} \M(v_*,I_*)^{1/2} \, \Bar{B}(v,v_*,I,I_*).
$$
It satisfies the expected $L^2$ property. 
\begin{proposition} \label{proposition:k1}
The kernel $\kappa_1$ belongs to $L^2 \left( \left(\R^3 \times \R_+ \right)^2, \; \ddd v \, \ddd \mu(I) \, \ddd v_* \, \ddd \mu(I_*) \right)$.
\end{proposition}

\begin{proof}
Applying the change of variables  $(v,v_*) \mapsto (v-v_*,v+v_*) = (V_*,\xi)$ and recalling the bound \eqref{eq:averagekernelbound} on $\Bar{B}$, we have
\begin{align*}
&\iint_{\R^3 \times \R_+} \iint_{\R^3 \times \R_+} \kappa_1(v,I,v_*,I_*)^2 \, \dd v \, \dd \mu(I) \, \dd v_* \, \dd \mu(I_*)\\
= \; &c^2 \iiiint_{(\R^3 \times \R_+)^2} e^{-\frac{|v|^2 + |v_*|^2}{2 k_B T_k}} \, e^{-\frac{I+I_*}{k_B T_i}} \, \bar{B}(v,v_*,I,I_*)^2 \, \dd v \, \dd \mu(I) \, \dd v_* \, \dd \mu(I_*)\\
\leq \; &C \left(\iint_{(\R^3)^2} e^{-\frac{|V_*|^2 + |\xi|^2}{4 k_B T_k}} \left(|V_*|^2 + |V_*|^{-1} \right)^2 \, \dd V_* \, \dd \xi \right) \left( \iint_{(\R_+)^2} e^{-\frac{I+I_*}{k_B T_i}} \, (I+I_*)^{\gamma}\, \dd \mu^{\otimes 2}(I,I_*) \right),
\end{align*}
where both previous integrals in the parentheses are finite.
\end{proof}

The compactness of $K_1$ is then a straightforward consequence of Proposition \ref{proposition:k1}, since $K_1$ is a Hilbert-Schmidt operator.

\subsection{Compactness of $K_2$}

With the same strategy as in \cite{MR3005625}, we prove the compactness of $K_2$, using the fact that $K_2$ is compact if it satisfies
\begin{itemize}
    \item a uniform decay at infinity, proven in Lemma \ref{lemma:boundK2},
    \item an equicontinuity property, proven in Proposition \ref{prop:equicontinuity}.
\end{itemize}
The consideration of resonant collision rules \eqref{eq:conservationresonant1}--\eqref{eq:conservationresonant2} along with the assumption \eqref{eq:B0majorbkbi} on the cross section  allows us to deal independently with a term in $v$ and another in $I$. While the term in the internal energy requires a specific treatment, the one in $v$ corresponds to the well-known monatomic case, whose compactness has already been proved \cite{grad2}. In our analysis, we shall use an alternative change of variable, whose properties have been detailed in Appendix~\ref{annex:monoatomic}.

\medskip

First, we intend to write $K_2$ as a kernel operator. For the sake of clarity, we now change the notation $J \equiv I'$. Recall that
$$
K_2  g (v,I) = \int_{\R^3 \times  (\R_+)^2 \times \Sb^2}   \, g(v',J) \, {\M(v',J)}^{-1/2} \M^{1/2} \M_* \, B(v,v_*,I,I_*,J,\sigma)  \, \dd v_*  \, \dd \mu^{\otimes 2}(I_*,J) \, \dd \sigma.
$$
We clearly have, almost everywhere,
$$
{\M(v',J)}^{-1/2} \M(v,I)^{1/2} \M \left(v_*,I_* \right) = M(v')^{-1/2} M(v)^{1/2} M(v_*) \, e^{\frac{J - I - 2 I_*}{2 k_B T_i}}.
$$
Thus $K_2 g(v,I)$ reads
\begin{multline*}
K_2  g (v,I)= \int_{\R^3 \times  (\R_+)^2 \times \Sb^2}   \, g(v',J) \, M(v')^{-1/2} M(v)^{1/2} M(v_*) \, e^{\frac{J-I - 2 I_*}{2 k_B T_i}} \\ B(v,v_*,I,I_*,J,\sigma)  \, \dd v_*  \, \dd \mu^{\otimes 2}(I_*,J) \, \dd \sigma.
\end{multline*}

We then perform several changes of variable in order to write $K_2$ as a kernel operator. With that purpose in mind, we recall the definition of the \emph{modified Bessel function of first kind of order $0$}, which we denote $\mathbb{I}_0$ throughout this paper, for all $X \in \R$,
\begin{equation} \label{eqdef:bessel}
    \mathbb{I}_0(X) = \frac{1}{2 \pi} \int_{0}^{2 \pi} \, e^{\, X \cos\theta } \, \dd \theta.
\end{equation}
We have the following proposition.
\begin{proposition}
We define, for $v$, $p \in \R^3$ and $I$, $I_*$, $J \in \R_+$,
\begin{align*}
\psi(v,p,I,I_*,J) = 8 \pi \, c \, \int_{\R_+}  &\exp\left( -\frac{r^2 + |v|^2|\sin(\widehat{v,p})|^2}{2 k_B T_k} \right) \, \mathbb{I}_0 \left( \frac{r \,  |v| \, |\sin(\widehat{v,p})| }{ k_B T_k} \right) \\ &\hspace{50pt} \times B_0\left( \sqrt{r^2 + |p|^2},\frac{|r^2 - |p|^2|}{r^2 + |p|^2},I,I_*,J \right) \, \frac{r}{\sqrt{r^2 + |p|^2}}     \,  \ddd r,
\end{align*}
as well as, for $v$, $\eta \in \R^3$ and $I$, $I_*$, $J \in \R_+$
\begin{equation} \label{eq:forme1kappa2}
\kappa_2(v,I,\eta,J) = \int_{\R_+}  e^{\frac{J-I - 2 I_*}{2 k_B T_i}} \, e^{ -\frac{ |\eta - v|^2}{8 k_B T_k}} \, e^{ -\frac{ 1}{8 k_B T_k} \, \frac{(|\eta|^2 - |v|^2)^2}{|\eta - v|^2} }  \, |\eta - v|^{-1} \, \psi(v,\eta - v,I,I_*,J)  \, \ddd \mu(I_*).
\end{equation}
Then we have, for almost every $(v,I)\in \R^3\times \R_+$, 
$$
K_2(g)(v,I) = \iint_{\R^3 \times \R_+}  g(\eta,J)  \,  \kappa_2(v,I,\eta,J) \, \ddd \eta \, \ddd \mu(J).
$$
\end{proposition}

\smallskip

\begin{proof}
The proof is completely similar to the one presentend in Appendix~\ref{annex:monoatomic}, with $I$, $I_*$ and $J$ treated as parameters. Indeed, we first notice that
\begin{multline*}
K_2  g (v,I)= \iint_{(\R^+)^2} e^{\frac{J-I - 2 I_*}{2 k_B T_i}}  \left[\iint_{\R^3 \times \Sb^2}   \, g(v',J) \, M(v')^{-1/2} M(v)^{1/2} M(v_*) \right. \\
\left.\phantom{\int}B_0 \left(|v-v_*|,\left|\cos\left( \widehat{v-v_*,\sigma} \right) \right|,I,I_*,J \right)  \, \dd v_*   \, \dd \sigma \right]\, \dd \mu^{\otimes 2}(I_*,J).
\end{multline*}
In the integral in $v_*$ and $\sigma$, the variables $I$, $I_*$ and $J$ are fixed parameters, so we can rewrite that integral as
$$
\iint_{\R^3 \times \Sb^2}   \,h(v') \, M(v')^{-1/2} M(v)^{1/2} M(v_*) \,  B^m \left(|v-v_*|,\left|\cos\left( \widehat{v-v_*,\sigma} \right) \right| \right)  \, \dd v_*   \, \dd \sigma,
$$
which corresponds to the monatomic operator $K_2^m$ defined in \eqref{eqdef:K2m}, with
$$
h = g(\cdot,J),\qquad B^m = B_0 \left(\cdot,\cdot,I,I_*,J \right).
$$
Applying Proposition \ref{proposition:k1kernel} yields the required result.
\end{proof}

The following lemma allows to upper-bound the kernel $\kappa_2$ by a tensorized product of two kernels, one dealing with the velocities, one with the internal energies.

\begin{lemma} \label{lem:kappaboundkappakkappai}
Let us define for $v$, $p$, $\eta \in \R^3$
\begin{align}
\psi_k(v,p) &= 8 \pi \, c \, \int_{\R_+}  \exp\left( -\frac{r^2 + |v|^2|\sin(\widehat{v,p})|^2}{2 k_B T_k} \right) \, \mathbb{I}_0 \left( \frac{r \,  |v| \, |\sin(\widehat{v,p})| }{ k_B T_k} \right) \nonumber \\ &\hspace{75pt} \times b_k\left( \sqrt{r^2 + |p|^2},\frac{|r^2 - |p|^2|}{r^2 + |p|^2} \right) \, \frac{r}{\sqrt{r^2 + |p|^2}}     \,  \ddd r, \label{eqdef:psik}\\
\kappa_k(v,\eta) &= e^{ -\frac{ |\eta - v|^2}{8 k_B T_k}} \, e^{ -\frac{ 1}{8 k_B T_k} \, \frac{(|\eta|^2 - |v|^2)^2}{|\eta - v|^2} }  \, |\eta - v|^{-1} \, \psi_k(v,\eta-v), \label{eqdef:kappak}
\end{align}
and, for $I,J \in \R_+$,
\begin{equation} \label{eqdef:kappai}
\kappa_i(I,J) = \int_{\R_+} \mathbf{1}_{J \leq I+I_*} \, e^{\frac{J-I - 2 I_*}{2 k_B T_i}} \,b_i(I,I_*) \, \ddd \mu(I_*) = \int_{(J-I)_+}^{\infty}  e^{\frac{J-I - 2 I_*}{2 k_B T_i}} \,b_i(I,I_*) \, \ddd \mu(I_*).
\end{equation}
Then
\begin{equation} \label{eq:kappaboundkappakkappai}
\kappa_2(v,I,\eta,J) \leq \kappa_k(v,\eta) \, \kappa_i(I,J).
\end{equation}
\end{lemma}

\begin{proof}
From \eqref{eq:B0majorbkbi}, we have
\begin{align*}
\psi(v,p,I,I_*,J) &\leq 8 \pi \, c \, \int_{\R_+}  \exp\left( -\frac{r^2 + |v|^2\sin(\widehat{v,p})^2}{2 k_B T_k} \right) \, \mathbb{I}_0 \left( \frac{r \,  |v| \, |\sin(\widehat{v,p})| }{ k_B T_k} \right) \\&\hspace{50pt}b_k\left( \sqrt{r^2 + |p|^2},\frac{|r^2 - |p|^2|}{r^2 + |p|^2} \right) \, \frac{r}{\sqrt{r^2 + |p|^2}}  \, b_i(I,I_*) \,  \mathbf{1}_{J \leq I+I_*} \, \dd r \\
&\leq   \psi_k(v,p) \, b_i(I,I_*)\, \mathbf{1}_{J \leq I+I_*}.
\end{align*}
Hence, we recover \eqref{eq:kappaboundkappakkappai}.
\end{proof}

As already stated, \eqref{eq:kappaboundkappakkappai} allows to separate the velocity part from the internal energy one. While a specific focus on the velocity part, corresponding to the monatomic case, is performed in Appendix~\ref{annex:monoatomic}, we now deal, in the following, with $\kappa_i$, the kernel dedicated to the internal energy part.

\begin{lemma} \label{lemma:boundb1}
Let us fix some $a \in \left(0, 1- \frac{\gamma}{2} \right)$. Then there exists $C > 0$ such that, if $I+I_* \leq 1$,
\begin{equation} \label{eq:lemmaboundbi1}
b_i(I,I_*)\leq C \, (I + I_*)^{\frac{\gamma}{2} - \beta_1 - 1},
\end{equation}
if $I+I_* \geq 1$,
\begin{equation} \label{eq:lemmaboundbi2}
b_i(I,I_*) \leq C \, (I + I_*)^{\frac{\gamma}{2} - \beta_2 + a - 1},
\end{equation}
\end{lemma}

\begin{proof}
We recall that, thanks to Assumption \eqref{eq:bimajor},
$$
b_i(I,I_*) \leq C \frac{(I+I_*)^{\gamma / 2}}{\mu[0,I+I_*]}.
$$
From Assumption \eqref{eqassumpmu1}, we have, for almost every $I$, $I_* \geq 0$ such that $I+I_* \leq 1$,
$$
\mu[0,I+I_*] = \int_{0}^{I+I_*} \dd \mu(J) \geq C \, \int_{0}^{I+I_*} J^{\beta_1} \dd J = C \, (I+I_*)^{\beta_1 + 1},
$$
so that, if $I+I_* \leq 1$,
$$
b_i(I,I_*) \leq C (I+I_*)^{\gamma / 2 - \beta_1 - 1}.
$$
On the other hand, from Assumption \eqref{eqassumpmu2}, we have, for almost every $I$, $I_* \geq 0$ such that $I+I_* \geq 1$,
$$
\mu([0,I+I_*]) \geq C \int_{0}^{1} J^{\beta_1} \, \dd J + C \int_{1}^{I+I_*} J^{\beta_2 - a} \, \dd J \geq  C' \, (I + I_*)^{\beta_2 - a + 1},
$$
so that, if $I+I_* \geq 1$,
$$
b_i(I,I_*) \leq C (I+I_*)^{\gamma / 2 - \beta_2 + a - 1}.
$$
\end{proof}

Lemma \ref{lemma:boundb1} paves the way to obtain relevant upper estimates on $\kappa_i$, as explained in the following lemma.

\begin{lemma} \label{lemma:boundkappai}
When $\gamma > 0$, there exists $C > 0$ such that, for almost every $I$, $J \geq 0$,
\begin{equation} \label{eq:lemmaboundkappaigammast0}
\kappa_i(I,J) \leq C \, e^{-\frac{|J-I|}{4 k_B T_i}} \, (1 + I)^{\frac{\gamma}{2} - \beta_2 +a - 1}.
\end{equation}
On the other hand, when $\gamma = 0$, for any $\alpha \in \left(0,1 - a \right)$, there exists $C_{\alpha} > 0$ such that, for almost every $I$, $J \geq 0$,
\begin{equation} \label{eq:lemmaboundkappaigamma0}
\kappa_i(I,J) \leq C_{\alpha} \, e^{-\frac{|J-I|}{4 k_B T_i}} \, I^{-\alpha} \, (1 + I)^{\alpha - \beta_2 + a - 1}.
\end{equation}
\end{lemma} 

\begin{proof}
In order to be able to use the polynomial bounds on $b_i$ and $\mu$, we divide the integral into three parts. We write
\begin{equation} \label{eq:kappaidecompose}
\kappa_i(I,J) = \mathcal{I}_1 + \mathcal{I}_2 + \mathcal{I}_3,
\end{equation}
with
\begin{align*}
     \mathcal{I}_1 &= \int_{(J-I)_+}^{\infty} \mathbf{1}_{I_* \geq 1} \,  e^{\frac{J-I - 2 I_*}{2 k_B T_i}} \,b_i(I,I_*) \, \dd \mu(I_*),\\
     \mathcal{I}_2 &= \int_{(J-I)_+}^{\infty} \mathbf{1}_{(1-I)_+ \leq I_* < 1} \,  e^{\frac{J-I - 2 I_*}{2 k_B T_i}} \,b_i(I,I_*) \, \dd \mu(I_*),\\
     \mathcal{I}_3 &= \int_{(J-I)_+}^{\infty} \mathbf{1}_{I_* < (1-I)_+} \,  e^{\frac{J-I - 2 I_*}{2 k_B T_i}} \,b_i(I,I_*) \, \dd \mu(I_*).
\end{align*}
Using the bounds \eqref{eq:lemmaboundbi2} on $b_i$ and \eqref{eqassumpmu2} on $\mu$ for $\mathcal{I}_1$, \eqref{eq:lemmaboundbi2} and \eqref{eqassumpmu1} for $\mathcal{I}_2$ and \eqref{eq:lemmaboundbi1} and \eqref{eqassumpmu1} for $\mathcal{I}_3$, we have
\begin{align*}
     \mathcal{I}_1 &\leq C \, \int_{(J-I)_+}^{\infty} \mathbf{1}_{I_* \geq 1} \,  e^{\frac{J-I - 2 I_*}{2 k_B T_i}} \,(I + I_*)^{\frac{\gamma}{2} - \beta_2 + a - 1} \, I_*^{\beta_2} \, \dd I_*,\\
     \mathcal{I}_2 &\leq C \, \int_{(J-I)_+}^{\infty} \mathbf{1}_{(1-I)_+ \leq I_* < 1} \,  e^{\frac{J-I - 2 I_*}{2 k_B T_i}} \,\,(I + I_*)^{\frac{\gamma}{2} - \beta_2 + a - 1} \, I_*^{\beta_1} \, \dd I_*,\\
     \mathcal{I}_3 &\leq C \, \int_{(J-I)_+}^{\infty} \mathbf{1}_{I_* < (1-I)_+} \,  e^{\frac{J-I - 2 I_*}{2 k_B T_i}} \,(I + I_*)^{\frac{\gamma}{2} - \beta_1 - 1} \, I_*^{\beta_1} \, \dd I_*.
\end{align*}
We first focus on $\mathcal{I}_1$. Since $\frac{\gamma}{2} - \beta_2 + a - 1 \leq 0$, we have, for all $I_* \geq 1$, that $(I + I_*)^{\frac{\gamma}{2} - \beta_2 + a - 1} \leq (1 + I)^{\frac{\gamma}{2} - \beta_2 + a - 1}$, thus we immediately have
\begin{equation} \label{eq:bound1I1}
\mathcal{I}_1 \leq C \, (1 + I)^{\frac{\gamma}{2} - \beta_2 + a - 1} \, \int_{(J-I)_+}^{\infty} \mathbf{1}_{I_* \geq 1} \,  e^{\frac{J-I - 2 I_*}{2 k_B T_i}} \, I_*^{\beta_2} \, \dd I_*.
\end{equation}
Note that there exists $C > 0$ such that, for all $I_* \geq 1$,
$$
e^{-\frac{ I_*}{k_B T_i}} \, I_*^{\beta_2} \leq C \, e^{-\frac{ 3 I_*}{4 k_B T_i}}.
$$
Hence, the integral in \eqref{eq:bound1I1} is upper-bounded by
$$
C \, e^{\frac{J-I}{2 k_B T_i}} \, \int_{(J-I)_+}^{\infty}  e^{-\frac{3 I_*}{4 k_B T_i}} \, \dd I_* \leq  C \,e^{\frac{J-I}{2 k_B T_i}} \, e^{-\frac{3 (J-I)_+ }{4 k_B T_i}} \leq C \, e^{-\frac{|J-I|}{4 k_B T_i}}.
$$
Thus, for all $I$, $J \geq 0$,
\begin{equation} \label{eq:boundfinalI1}
\mathcal{I}_1 \leq C \, (1 + I)^{\frac{\gamma}{2} - \beta_2 + a - 1} \, e^{-\frac{|J-I|}{4 k_B T_i}}.
\end{equation}
We now focus on $\mathcal{I}_2$. Since $\beta_1 \geq 0$,  $I_*^{\beta_1} \leq 1$ when $I_* \leq 1$. Moreover, since $\frac{\gamma}{2} - \beta_2 + a - 1 \leq 0$, we also have $(I_*+I)^{\frac{\gamma}{2} - \beta_2 + a - 1} \leq I^{\frac{\gamma}{2} - \beta_2 + a - 1}$. Then
\begin{equation}  \label{eq:bound1I2}
\mathcal{I}_2 \leq C \, I^{\frac{\gamma}{2} - \beta_2 + a - 1} \, \int_{(J-I)_+}^{\infty}   e^{\frac{J-I - 2 I_*}{2 k_B T_i}}  \, \dd I_* \leq C \, I^{\frac{\gamma}{2} - \beta_2 + a - 1} \,
e^{-\frac{|J-I|}{4 k_B T_i}}.
\end{equation}
Moreover, for all $I \in [0,1)$,
\begin{align}
\mathcal{I}_2 &\leq C \, \int_{(J-I)_+}^{\infty} \mathbf{1}_{1-I \leq I_* < 1} \,  e^{\frac{J-I - 2 I_*}{2 k_B T_i}} \,\,(I + I_*)^{\frac{\gamma}{2} - \beta_2 + a - 1} \, I_*^{\beta_1} \, \dd I_* \nonumber \\
&\leq C \, \int_{(J-I)_+}^{\infty} \mathbf{1}_{1 \leq I + I_*} \,  e^{\frac{J-I - 2 I_*}{2 k_B T_i}} \,(I + I_*)^{\frac{\gamma}{2} - \beta_2 + a - 1} \, I_*^{\beta_1} \, \dd I_*  \nonumber \\
&\leq C \, \int_{(J-I)_+}^{\infty}  e^{\frac{J-I - 2 I_*}{2 k_B T_i}} \, I_*^{\beta_1} \, \dd I_* \leq C' \, e^{-\frac{|J-I|}{4 k_B T_i}}. \label{eq:bound2I2}
\end{align}
Combining \eqref{eq:bound1I2}--\eqref{eq:bound2I2} yields, for all $I$, $J \geq 0$,
\begin{equation} \label{eq;boundfinalI2}
\mathcal{I}_2 \leq C \, (1 + I)^{\frac{\gamma}{2} - \beta_2 + a - 1} \, e^{-\frac{|J-I|}{4 k_B T_i}}.
\end{equation}
Finally, let us deal with $\mathcal{I}_3$. We note that $\mathcal{I}_3$ vanishes when $I \geq 1$ or when we simultaneously have $I < 1$ and $J \geq 1$. When $I < 1$ and $J < 1$, since $(I + I_*)^{-\beta_1} \, I_*^{\beta_1} \leq 1$ and $J-I - 2 I_* \leq 0$ when $I_* \geq (J-I)_+$,
\begin{align*}
    \mathcal{I}_3 \leq C \, \int_{(J-I)_+}^{\infty} \mathbf{1}_{I_* + I < 1} \, (I+I_*)^{\frac{\gamma}{2}  - 1} \, \dd I_*.
\end{align*}
In order to conclude, we have to discuss the cases $\gamma > 0$ and $\gamma = 0$ separately.

\medskip

\noindent \emph{Case $\gamma > 0$.} Since $ -1 < \gamma / 2 - 1 < 0$, we have $(I+I_*)^{\gamma/2 - 1} \leq I_*^{\gamma/2 - 1}$ and
\begin{equation} \label{eq:bound1I3}
\mathcal{I}_3 \leq C \int_{0}^1 I_*^{\gamma/2 - 1} \, \dd I_* = C.
\end{equation}
\smallskip

\noindent \emph{Case $\gamma = 0$.} For any $\alpha \in \left(0,1 - a \right)$ we have $(I+I_*)^{-1} \leq I^{-\alpha} \, I_*^{\alpha - 1}$ and
\begin{equation} \label{eq:bound2I3}
\mathcal{I}_3 \leq C  \, I^{-\alpha} \int_{0}^1 I_*^{\alpha - 1} \, \dd I_* = C_{\alpha} \, I^{-\alpha}.
\end{equation}
Taking into account \eqref{eq:boundfinalI1} and \eqref{eq;boundfinalI2}--\eqref{eq:bound2I3} in \eqref{eq:kappaidecompose} allows to end the proof.
\end{proof}

The pointwise estimates of Lemma~\ref{lemma:boundkappai} leads to $L^1$ estimates with respect to one variable, $I$ and $J$, successively.
\begin{corollary} \label{corollary:kappai}
When $\gamma > 0$, there exists $C > 0$ such that
\begin{align}
&\int_{\R_+} \kappa_i(I,J) \, \ddd \mu(J) \leq C \, (1 + I)^{\frac{\gamma}{2}+a-1}, \qquad  \text{for a.e. } I \geq 0, \label{eq:corollarykappaist01} \\
&\int_{\R_+} \kappa_i(I,J) \, \ddd \mu(I) \leq C, \qquad \qquad \qquad \qquad  \text{for a.e. } J \geq 0. \label{eq:corollarykappaist02}
\end{align}
When $\gamma = 0$, for any $\alpha \in \left(0,\frac{1-a}{2} \right)$, there exists $C_{\alpha}>0$ such that
\begin{align}
&\int_{\R_+} \kappa_i(I,J) \, \ddd \mu(J) \leq C_{\alpha} \, I^{-\alpha} \, (1 + I)^{\alpha+a-1}, \qquad   \text{for a.e. } I \geq 0, \label{eq:corollarykappai01}\\
&\int_{\R_+} I^{-\alpha} \, \kappa_i(I,J) \, \ddd \mu(I) \leq C_{\alpha},  \qquad \qquad \qquad \qquad   \text{for a.e. } J \geq 0. \label{eq:corollarykappai02}
\end{align}

\end{corollary}

\begin{proof}
Again we have to distinguish the cases $\gamma > 0$ and $\gamma = 0$.

\medskip

\noindent \emph{Case $\gamma > 0$.} We combine the bound \eqref{eq:lemmaboundkappaigammast0} on $\kappa_i$  given in Lemma \ref{lemma:boundkappai}, and \eqref{eq:expdeltafrac} in Corollary \ref{corollary:expdeltafrac} with $r = -\beta_2$ and $s = \frac{1}{8 k_B T_i}$, to obtain
$$
\kappa_i(I,J) \leq C \,  e^{-\frac{|J-I|}{4 k_B T_i}} \, (1 + I)^{\frac{\gamma}{2} - \beta_2 + a - 1} \leq C \, e^{-\frac{|J-I|}{8 k_B T_i}} \, (1 + I)^{\frac{\gamma}{2} + a - 1} \, (1 + J)^{-\beta_2}.
$$
Using the hypotheses \eqref{eqassumpmu1}--\eqref{eqassumpmu2} on $\mu$ and noticing that, for all $J \in [0,1]$, $(1 + J)^{-\beta_2} \, J^{\beta_1} \leq 1$ and that, for all $J \geq 1$, $(1 + J)^{-\beta_2} \, J^{\beta_2} \leq 1$,
\begin{align*}
 \int_{\R_+} \kappa_i(I,J) \, \dd \mu(J) &\leq C \, \int_{1}^{\infty} e^{-\frac{|J-I|}{8 k_B T_i}} \, (1 + I)^{\frac{\gamma}{2} + a - 1} \, (1 + J)^{-\beta_2} \, J^{\beta_2} \, \dd J \\
 &+ C \, \int_0^1 e^{-\frac{|J-I|}{8 k_B T_i}} \, (1 + I)^{\frac{\gamma}{2} + a - 1} \, (1 + J)^{-\beta_2} \, J^{\beta_1} \, \dd J \\
 &\leq C \, (1 + I)^{\frac{\gamma}{2} + a - 1} \, \int_{\R_+} e^{-\frac{|J-I|}{8 k_B T_i}} \, \dd J \leq  C \, (1 + I)^{\frac{\gamma}{2} + a - 1} \, \int_{\R} e^{-\frac{|J|}{8 k_B T_i}} \, \dd J \\
 &\leq C \, (1 + I)^{\frac{\gamma}{2} + a - 1}.
\end{align*}
For the other inequality, from \eqref{eq:lemmaboundkappaigammast0} we have
$$
\kappa_i(I,J)  \leq C \, e^{-\frac{|J-I|}{8 k_B T_i}} \, (1 + I)^{\frac{\gamma}{2} - \beta_2 + a - 1} \leq C \, e^{-\frac{|J-I|}{8 k_B T_i}}  \, (1 + I)^{-\beta_2},
$$
because $\frac{\gamma}{2} + a - 1 \leq 0$. It comes that, again from \eqref{eqassumpmu1}--\eqref{eqassumpmu2},
\begin{align*}
 \int_{\R_+} \kappa_i(I,J) \, \dd \mu(I) &\leq C \left(  \int_0^1 e^{-\frac{|J-I|}{8 k_B T_i}} \, (1 + I)^{-\beta_2} \, I^{\beta_1} \, \dd I + \int_{1}^{\infty} e^{-\frac{|J-I|}{8 k_B T_i}} \, (1 + I)^{-\beta_2} \, I^{\beta_2} \, \dd I\right) \\
 &\leq  C \, \int_{\R} e^{-\frac{|I|}{8 k_B T_i}} \, \dd I,
\end{align*}
which does not depend on $J$.

\medskip

\noindent \emph{Case $\gamma = 0$.}  We combine the bound \eqref{eq:lemmaboundkappaigamma0} on $\kappa_i$  given in Lemma \ref{lemma:boundkappai}, and \eqref{eq:expdeltafrac} with $r = -\beta_2$ and $s = \frac{1}{8 k_B T_i}$, to obtain that, for any $\alpha \in \left(0,\frac{1-a}{2} \right)$,
$$
\kappa_i(I,J)  \leq C \,  e^{-\frac{|J-I|}{4 k_B T_i}} \, I^{-\alpha} \, (1 + I)^{\alpha - \beta_2 + a - 1} \leq C\, e^{-\frac{|J-I|}{8 k_B T_i}} \, I^{-\alpha}  \, (1 + I)^{\alpha + a - 1} \, (1 + J)^{-\beta_2}, \; \; \; \text{for a.e. } I, J \geq 0,
$$
from which we get, with the same computations as in the case $\gamma > 0$,
$$
\int_{\R_+} \kappa_i(I,J) \, \dd \mu(J) \leq C_{\alpha} \, I^{-\alpha} \, (1 + I)^{\alpha + a - 1}, \qquad \text{for a.e. } I \geq 0.
$$
On the other hand, once again using Lemma \ref{lemma:expdeltafrac}, we have
$$
\kappa_i(I,J) \leq C_{\alpha} \, I^{-\alpha} \, e^{-\frac{|J-I|}{8 k_B T_i}} \, (1 + J)^{\alpha + a - 1} \, (1 + I)^{-\beta_2} \leq C_{\alpha} \, I^{-\alpha} \, e^{-\frac{|J-I|}{8 k_B T_i}} \, (1 + I)^{-\beta_2} 
$$
because $\alpha + a - 1 \leq 0$. Recalling that $-2 \alpha > -1$, we get
\begin{multline*}
 \int_{\R_+} I^{-\alpha} \,  \kappa_i(I,J) \, \dd \mu(I) \\
\leq C \, \int_1^{\infty} I^{- 2 \alpha} \, e^{-\frac{|J-I|}{8 k_B T_i}} \, (1 + I)^{-\beta_2} \, I^{\beta_2} \, \dd I + C \, \int_0^1 I^{- 2 \alpha} \, e^{-\frac{|J-I|}{8 k_B T_i}} \, (1 + I)^{-\beta_2} \, I^{\beta_1} \, \dd I \\
 \leq C \, \int_{\R_+}  I^{- 2 \alpha} \, e^{-\frac{|I-J|}{8 k_B T_i}} \, \dd I \leq C \, \left(\int_{0}^1  I^{- 2 \alpha}  \dd I + \int_{1}^{\infty}  e^{-\frac{|I-J|}{8 k_B T_i}} \, \dd I \right) \\
\leq C \, \left(\frac{1}{1-2\alpha}+ \int_{\R}  e^{-\frac{|I|}{8 k_B T_i}} \, \dd I \right),
\end{multline*}
which does not depend on $J$.
\end{proof}

We are now in a position to deduce integrability properties of $\kappa_i$ in $L^2$. 
\begin{lemma} \label{lemma:kappail2loc}
The kernel $\kappa_i$ belongs to $L^2_{loc}\left(\R_+, \,  \ddd \mu(I) ; \, L^2 \left(  \R_+, \,  \ddd \mu(J) \right) \right)$.
\end{lemma}

\begin{proof}
The proof is similar to the one of Corollary \ref{corollary:kappai}. We distinguish the cases $\gamma > 0$ and $\gamma = 0$.

\medskip

\noindent \emph{Case $\gamma > 0$.} We combine the bound \eqref{eq:lemmaboundkappaigammast0} on $\kappa_i$  given by Lemma \ref{lemma:boundkappai}, and \eqref{eq:expdeltafrac} with $r = - \frac{\beta_2}{2}$ and $s = \frac{1}{8 k_B T_i}$, to obtain, for almost every $I$ and $J$, 
$$
\kappa_i(I,J) \leq C \, e^{-\frac{|J-I|}{8 k_B T_i}} \, (1 + I)^{\frac{\gamma}{2} - \frac{\beta_2}{2} + a - 1} \, (1 + J)^{-\frac{\beta_2}{2}} \leq C \, e^{-\frac{|J-I|}{8 k_B T_i}} \, (1 + J)^{-\frac{\beta_2}{2}},
$$
because $\frac{\gamma}{2} - \frac{\beta_2}{2} + a - 1 \leq 0$. Hence, we can write
\begin{align*}
 \int_{\R_+} \kappa_i(I,J)^2 \, \dd \mu(J) &\leq C \left(  \int_0^1 e^{-\frac{|J-I|}{4 k_B T_i}} \, (1 + J)^{-\beta_2} \, J^{\beta_1} \, \dd J + \int_{1}^{\infty} e^{-\frac{|J-I|}{4 k_B T_i}} \, (1 + J)^{-\beta_2} \, J^{\beta_2} \, \dd J\right) \\
 &\leq  C \, \int_{\R} e^{-\frac{|J|}{4 k_B T_i}} \, \dd J = C.
\end{align*}
Assumptions \eqref{eqassumpmu1}--\eqref{eqassumpmu2} on $\mu$ imply that constants lie in $L^1_{loc} \left( \R_+, \,  \dd \mu(I) \right)$, which allows to conclude in this case.

\medskip

\noindent \emph{Case $\gamma = 0$.} We combine the bound \eqref{eq:lemmaboundkappaigamma0} on $\kappa_i$  given in Lemma \ref{lemma:boundkappai}, and \eqref{eq:expdeltafrac} with $r = -\frac{\beta_2}{2}$ and $s = \frac{1}{8 k_B T_i}$, to obtain that, for any $\alpha \in \left(0,\frac{1-a}{2} \right)$,
$$
\kappa_i(I,J) \leq C_{\alpha} \, I^{-\alpha} \, e^{-\frac{|J-I|}{8 k_B T_i}} \, (1 + I)^{\alpha - \frac{\beta_2}{2} - 1 + a} \, (1 + J)^{-\frac{\beta_2}{2}} \leq C_{\alpha} \, I^{-\alpha} \, e^{-\frac{|J-I|}{8 k_B T_i}}  \, (1 + J)^{-\frac{\beta_2}{2}},
$$
because $\frac{\gamma}{2} - \frac{\beta_2}{2} + a - 1 \leq 0$. Hence, going through the same computations as in the case $\gamma > 0$, we get, for almost every $I \geq 0$,
$$
\int_{\R_+} \kappa_i(I,J)^2 \, \dd \mu (J) \leq C_{\alpha} \, I^{-2\alpha} .
$$
Since $-2\alpha > -1$, Assumptions~\eqref{eqassumpmu1}--\eqref{eqassumpmu2} on $\mu$ imply that $I \mapsto I^{-2\alpha}$ belongs to $L^1_{loc} \left( \R_+, \,  \dd \mu(I) \right)$, which ends the proof.
\end{proof}

\begin{remark}
We can in fact prove that, if $\gamma < 1$, then $\kappa_i \in L^2\left(\R_+, \,  \ddd \mu(I) ; \, L^2 \left(  \R_+, \,  \ddd \mu(J) \right) \right)$.
\end{remark}

We now combine the results on $\kappa_i$ along with the ones on $\kappa_k$, obtained from the monatomic case studied in Appendix~\ref{annex:monoatomic}.

\begin{lemma} \label{lemma2:boundkappa}
When $\gamma > 0$, there exists $C > 0$ such that
\begin{align}
 &\iint_{\R^3 \times \R_+} \kappa_2(v,I,\eta,J) \,  \ddd \eta \, \ddd \mu(J) \leq  \frac{C}{(1 + I + |v|)^{1 - \frac{\gamma}{2} - a} },   \qquad    \text{for a.e. } (v,I) \in \R^3 \times \R_+,   \label{eq:lemmakappast01} \\
 &\iint_{\R^3 \times \R_+} \kappa_2(v,I,\eta,J) \, \ddd v \, \ddd \mu(I)  \leq C,   \qquad \qquad \qquad \qquad \quad \; \; \, \text{for a.e. } (\eta,J) \in \R^3 \times \R_+.      \label{eq:lemmakappast02} 
\end{align}
When $\gamma = 0$, for any $\alpha \in \left(0,\frac{1-a}{2} \right)$, there exists $C_{\alpha}>0$ such that
\begin{align}
&\iint_{\R^3 \times \R_+} \kappa_2(v,I,\eta,J) \,  \ddd \eta \, \ddd \mu(J) \leq  \frac{C_{\alpha} \, I^{-\alpha}}{(1 + I + |v|)^{1 - \alpha - a} }, \qquad \text{for a.e. } (v,I) \in \R^3 \times \R_+,   \label{eq:lemmakappa01} \\
&\iint_{\R^3 \times \R_+} I^{-\alpha} \, \kappa_2(v,I,\eta,J) \, \ddd v \, \ddd \mu(I)  \leq C_{\alpha},   \qquad \qquad \qquad  \quad \, \text{for a.e. } (\eta,J) \in \R^3 \times \R_+. \label{eq:lemmakappa02} 
\end{align}
\end{lemma}

\begin{proof}
We distinguish the cases $\gamma > 0$ and $\gamma = 0$.

\medskip

\noindent \emph{Case $\gamma > 0$.} From the bounds  \eqref{eq:kappakbound1} and \eqref{eq:corollarykappaist01}, also recalling that $\frac{\gamma}{2} + a - 1 \in (-1,0)$, we get
\begin{align*}
\iint_{\R^3 \times \R_+} \kappa_2(v,I,\eta,J) \,  \dd \eta \, \dd \mu(J)  &\leq \left(  \int_{\R^3} \kappa_k(v,\eta) \, \dd \eta \right)  \left( \int_{\R_+} \kappa_i(I,J) \, \dd \mu(J)  \right) \\
&\leq  \frac{C}{(1 + |v|) \, (1 + I)^{1 - \frac{\gamma}{2} - a}} \leq \frac{C}{(1 + |v| + I)^{1 - \frac{\gamma}{2} - a} }.
\end{align*}
Similarly, using the bounds \eqref{eq:kappakbound2} and \eqref{eq:corollarykappaist02}, we get \eqref{eq:lemmakappast02}.

\medskip

\noindent \emph{Case $\gamma = 0$.} From   \eqref{eq:kappakbound1}  and   \eqref{eq:corollarykappai01}, for any $\alpha \in \left(0,\frac{1-a}{2} \right)$, remarking that $\alpha + a - 1 \in (-1,0)$, we have
\begin{align*}
\iint_{\R^3 \times \R_+} \kappa_2(v,I,\eta,J) \,  \dd \eta \, \dd \mu(J)  &\leq \left(  \int_{\R^3} \kappa_k(v,\eta) \, \dd \eta \right)  \left( \int_{\R_+}  \kappa_i(I,J) \, \dd \mu(J)  \right) \\
&\leq  \frac{C_{\alpha} \, I^{-\alpha}}{(1 + |v|) \, (1 + I)^{1 - \alpha - a}} \leq \frac{C_{\alpha} \, I^{-\alpha}}{(1 + |v| + I)^{1 - \alpha - a} },
\end{align*}
and again, similarly to the case $\gamma> 0$, we obtain \eqref{eq:lemmakappa02}.
\end{proof}

The following lemma is then a straightforward consequence of Lemmas~\ref{lem:kappaboundkappakkappai}, \ref{lemma:kappail2loc} and \ref{prop:kappa0L2loc}.
\begin{lemma} \label{lemma:kappaL2}
The kernel $\kappa_2$ belongs to $L^2_{loc}\left( \R^3 \times \R_+, \, \ddd v \, \ddd \mu(I) ; \, L^2 \left(\R^3 \times \R_+, \, \ddd \eta \, \ddd \mu(J) \right) \right)$.
\end{lemma}

All norms being equivalent in finite dimension, we can choose, on the space $\R^3 \times \R$, the norm $|(v,I)|_{\R^4} = |v| + |I|$, where $|v|$ is the Euclidean norm of $v$ in $\R^3$. In the following, $\B_R$ is defined for any $R > 0$ by $\B_R = \{(v,I) \in \R^3 \times \R_+ \text{ s.t. } |v| + I < R \}$.

\begin{lemma} \label{lemma:boundK2}
When $\gamma > 0$, there exists $C > 0$ such that, for all $g \in L^2 \left(\R^3 \times \R_+, \, \ddd \eta \, \ddd \mu(J) \right)$ and any $R> 0$,
\begin{equation} \label{eq:boundK2st0}
\|K_2 g \|^2_{L^2 \left(\B_R^c, \, \ddd v \, \ddd \mu(I) \right)} \leq \frac{C}{(1 + R)^{1 - \frac{\gamma}{2} - a}} \, \| g\|^2_{L^2\left(\R^3 \times \R_+, \, \ddd \eta \, \ddd \mu(J) \right)}.
\end{equation}
When $\gamma = 0$, for any $\alpha \in \left(0,\frac{1-a}{2} \right)$, there exists $C_{\alpha}>0$ such that, for any $R> 0$, and all $g \in L^2 \left(\R^3 \times \R_+, \, \ddd \eta \, \ddd \mu(J) \right)$,
\begin{equation} \label{eq:boundK20}
\|K_2 g \|^2_{L^2 \left(\B_R^c, \, \ddd v \, \ddd \mu(I) \right)} \leq \frac{C_{\alpha}}{(1 + R)^{1 - \alpha - a}} \, \| g\|^2_{L^2\left(\R^3 \times \R_+, \, \ddd \eta \, \ddd \mu(J) \right)}.
\end{equation}
\end{lemma}

\begin{proof}
Let $R > 0$. We distinguish the cases $\gamma > 0$ and $\gamma = 0$.

\medskip

\noindent \emph{Case $\gamma > 0$.} Using the Cauchy-Schwarz inequality and \eqref{eq:lemmakappast01}--\eqref{eq:lemmakappast02}, we get
\begin{align*}
&\|K_2 g \|^2_{L^2 \left(\B_R^c, \, \dd v \, \dd \mu(I) \right)} = \iint_{|v| + I \geq R} \left( \iint_{\R^3 \times \R_+} g(\eta,J) \, \kappa_2(v,I,\eta,J) \, \dd \eta \,\dd \mu(J) \right)^2  \, \dd v \, \dd \mu(I)  \\
&\leq \iint_{|v| + I \geq R}  \left( \int_{\R^3 \times \R_+} g(\eta,J)^2 \, \kappa_2(v,I,\eta,J) \, \dd \eta \,\dd \mu(J) \right) \left( \iint_{\R^3 \times \R_+}  \kappa_2(v,I,\eta,J) \, \dd \eta \,\dd \mu(J) \right) \, \dd v \, \dd \mu(I)  \\
&\leq  \iint_{|v| + I \geq R} \left( \iint_{\R^3 \times \R_+}  g(\eta,J)^2 \, \kappa_2(v,I,\eta,J) \, \dd \eta \,\dd \mu(J) \right) \frac{C}{(1 + |v| + I)^{1 - \frac{\gamma}{2} - a}}  \, \dd v \, \dd \mu(I)  \\
&\leq \frac{C}{(1 + R)^{1 - \frac{\gamma}{2} - a}} \, \iint_{\R^3 \times \R_+}  g(\eta,J)^2 \, \left(  \iint_{\R^3 \times \R_+} \kappa_2(v,I,\eta,J) \, \dd v \,  \dd \mu(I)  \right)  \, \dd \eta \,\dd \mu(J) \\
&\leq \frac{C}{(1 + R)^{1 - \frac{\gamma}{2} - a}} \, \| g\|^2_{L^2\left(\R^3 \times \R_+, \, \dd \eta \, \dd \mu(J) \right)},
\end{align*}
where $C$ depends neither on $R$ nor on $g$.

\medskip

\noindent \emph{Case $\gamma = 0$.} For any $\alpha \in \left(0,\frac{1-a}{2} \right)$, we get, using the Cauchy-Schwarz inequality and \eqref{eq:lemmakappa01}--\eqref{eq:lemmakappa02},
\begin{align*}
&\|K_2 g \|^2_{L^2 \left(\B_R^c, \, \dd v \, \dd \mu(I) \right)} = \iint_{|v| + I \geq R} \left( \iint_{\R^3 \times \R_+} g(\eta,J) \, \kappa_2(v,I,\eta,J) \, \dd \eta \,\dd \mu(J) \right)^2  \, \dd v \, \dd \mu(I)  \\
&\leq \iint_{|v| + I \geq R}  \left( \iint_{\R^3 \times \R_+} g(\eta,J)^2 \, \kappa_2(v,I,\eta,J) \, \dd \eta \,\dd \mu(J) \right) \left( \iint_{\R^3 \times \R_+}  \kappa_2(v,I,\eta,J) \, \dd \eta \,\dd \mu(J) \right) \, \dd v \, \dd \mu(I)  \\
&\leq \, \iint_{|v| + I \geq R}  \left( \iint_{\R^3 \times \R_+} g(\eta,J)^2 \, \kappa_2(v,I,\eta,J) \, \dd \eta \,\dd \mu(J) \right) \frac{C_{\alpha} \, I^{-\alpha}}{(1 + |v| + I)^{1 - \alpha - a}}  \, \dd v \, \dd \mu(I)  \\
&\leq \frac{C_{\alpha}}{(1 + R)^{1 - \alpha - a}} \, \iint_{\R^3 \times \R_+} g(\eta,J)^2 \, \left(  \iint_{\R^3 \times \R_+}  I^{-\alpha}  \,  \kappa_2(v,I,\eta,J) \, \dd v \,  \dd \mu(I)  \right)  \, \dd \eta \,\dd \mu(J) \\
&\leq \frac{C_{\alpha}}{(1 + R)^{1 - \alpha - a}} \, \| g\|^2_{L^2\left(\R^3 \times \R_+, \, \dd \eta \, \dd \mu(J) \right)},
\end{align*}
which concludes the proof.
\end{proof}

In the following, we denote by $\tau_{(w,H)}$ the operator of translation by $(w,H)\in \R^3 \times \R_+$.

\begin{proposition} \label{prop:equicontinuity}
For any $\varepsilon > 0$, there exists $s > 0$ such that, for any $g \in L^2 \left(\R^3 \times \R_+, \, \ddd \eta \, \ddd \mu(J) \right)$,
$$
\left\|\Big((\tau_{(w,H)} - \emph{Id})K_2 \Big)g \right\|_{L^2 \left(\R^3 \times \R_+, \, \ddd v \, \ddd \mu(I) \right)} \leq \varepsilon \, \| g\|_{L^2\left(\R^3 \times \R_+, \, \ddd \eta \, \ddd \mu(J) \right)}, \qquad \text{for a.e. } (w,H) \in \B_s.
$$
\end{proposition}

\begin{proof}
Let $R> 0$. We have, for almost every $(w,H) \in \B_R$,
\begin{equation} \label{eq:tau1}
    \begin{split}
     &\left\|\Big((\tau_{(w,H)} - \text{Id})K_2 \Big)g \right\|^2_{L^2 \left(\R^3 \times \R_+, \, \dd v \, \dd \mu(I) \right)}\\
&\leq \iint_{\B_{2 R}} \left(K_2 g(v+w,I+H) - K_2 g(v,I) \right)^2 \, \dd v \, \dd \mu(I) + \iint_{\B_R^c} K_2 g(v,I)^2 \, \dd v \, \dd \mu(I).   
    \end{split}
\end{equation}
Let us now distinguish the cases $\gamma > 0$ and $\gamma = 0$.

\medskip

\noindent \emph{Case $\gamma > 0$.} Equation \eqref{eq:boundK2st0} ensures that the second integral is bounded by 
$$
\frac{C}{(1 + R)^{1 - \frac{\gamma}{2} - a}} \, \| g\|^2_{L^2\left(\R^3 \times \R_+, \, \dd \eta \, \dd \mu(J) \right)}.
$$
We choose $R$ such that $\frac{C}{(1 + R)^{1 - \frac{\gamma}{2} - a}} \leq \frac{\varepsilon^2}{2}$. 

\medskip

\noindent \emph{Case $\gamma = 0$.} We choose $\alpha \in \left(0,\frac{1-a}{2} \right)$. Equation \eqref{eq:boundK20} ensures that the second integral is bounded by 
$$
\frac{C_{\alpha}}{(1 + R)^{1 - \alpha - a}} \, \| g\|^2_{L^2\left(\R^3 \times \R_+, \, \dd \eta \, \dd \mu(J) \right)}.
$$
We choose $R$ such that $\frac{C}{(1 + R)^{1 - \alpha - a}} \leq \frac{\varepsilon^2}{2}$. 

\bigskip

\noindent Let us conclude, now that $R$ is fixed. We notice that the first integral in \eqref{eq:tau1} can be upper-bounded by
$$
\| g\|^2_{L^2\left(\R^3 \times \R_+, \, \dd \eta \, \dd \mu(J) \right)} \, \iint_{\B_{2R}} \iint_{\R^3 \times \R_+} \left(\kappa_2(v+w,I+H,\eta,J) - \kappa_2(v,I,\eta,J) \right)^2 \, \dd \eta \, \dd \mu(J) \, \dd v \, \dd \mu(I).
$$
Thanks to Lemma \ref{lemma:kappaL2}, we know that $\kappa_2 \in L^2_{loc}\left( \R^3 \times \R_+, \, \dd v \, \dd \mu(I) ; \, L^2 \left(\R^3 \times \R_+, \, \dd \eta \, \dd \mu(J) \right) \right)$, consequently, there exists $s \in (0,R)$ such that, if $|w| + H \leq s$, then
$$
\iint_{\B_{2R}} \iint_{\R^3 \times \R_+} \left(\kappa_2(v+w,I+H,\eta,J) - \kappa_2(v,I,\eta,J) \right)^2 \, \dd \eta \, \dd \mu(J) \, \dd v \, \dd \mu(I) \leq \frac{\varepsilon^2}{2},
$$
which ends the proof.
\end{proof}

\subsection{Conclusion}

We now know that both $K_1$ and $K_2$ own the expected compactness property. So does $K_3$, as we already stated it, because $K_2$ and $K_3$ are directly related thanks to a change of variable. As a sum of three compact operators, $K$ is also compact, and that concludes the proof of Theorem~\ref{theorem}.


\appendix

\section{Grad's proof revisited} \label{annex:monoatomic}

In this appendix, we aim to present a variant proof of Grad's, involving a geometric interpretation of the change of variables to obtain the kernel form, which allows to recover the compactness property.

\smallskip

Let us briefly recall the monatomic case. In this context, no internal energies are involved, and the collision rules are given by \eqref{eq:primevelocities}. We assume that the cross section $B^m$ only depends on $|v-v_*|$ and $\left|\cos\left(\widehat{v_*-v,\sigma} \right) \right|$, in the same way as in \eqref{eq:bkmajor}, {\it i.e.} there exists $\delta_1 \in [0,1)$, $\delta_2 \in [0,\frac12)$ and a constant $C$ such that for almost every $\rho \geq 0$ and $\theta \in [0,2 \pi]$,
\begin{equation} \label{eq:bmmajor}
B^m(\rho, |\cos \theta |) \leq C \left( |\sin \theta | \left( \rho^2 + \rho^{-1} \right) + \rho + \rho^{-\delta_1} + |\sin \theta |^{-\delta_2} \right).
\end{equation}
The difficult part of the compactness proof in the monatomic case mainly lies in the study of the operator $K^m_2$ defined, for all measurable function $h$ for which it makes sense and for almost every $v \in \R^3$, by
\begin{equation} \label{eqdef:K2m}
\begin{split}
  K^m_2 h(v) = \int_{\R^3} \int_{\Sb^2} h \left(\frac{v+v_*}{2} + \frac{|v-v_*|}{2} \sigma \right) \,  M\left(\frac{v+v_*}{2} + \frac{|v-v_*|}{2} \sigma \right)^{-1/2} M(v)^{1/2} M(v_*) \\
\phantom{\int} B^m \left(|v-v_*|, \left|\cos\left(\widehat{v_*-v,\sigma} \right) \right| \right) \, \dd \sigma \, \dd v_*,
  \end{split}
\end{equation}
where we recall that $\displaystyle M(v) = c \, e^{-\frac{|v|^2}{2 k_B T_k}}$ for some constant $c > 0$.

\medskip

It is worth noting that, thanks to an interpolation argument, the cross section $|\sin \theta |^{\alpha} \, \rho^{1 + \alpha}$, $\alpha \in [0,1]$, is also covered by Assumption~\eqref{eq:bmmajor}. Indeed,
$$
|\sin \theta |^{\alpha} \, \rho^{1 + \alpha} = \left(|\sin\theta |  \rho^2 \right)^{\alpha} \, \rho^{1-\alpha} \leq \alpha \, |\sin\theta |  \rho^2 + (1 - \alpha) \, \rho \leq |\sin\theta |  \rho^2 +  \rho.
$$
Similarly, for any $\alpha \in (0,1)$, $\gamma_1 \in (-\alpha,\alpha]$ and $\gamma_2 \in [0,\frac{1-\alpha}{2})$, the cross section $|\sin \theta |^{-\gamma_2} \, \rho^{\gamma_1}$ is also covered.

\subsection{Alternative change of variables} \label{subsectionzA}

Let us first discuss the crucial change of variables which is used later on to obtain the kernel form of the operator $K^m_2$. The idea is to change a couple of angles in the two-dimensional sphere $(\Theta, \sigma)$ into a couple composed of an element of the three-dimensional open ball and an angle in the one-dimensional circle $(z,A)$. 

On the one hand, if we take $\Theta \in \Sb^2$ and $\sigma \in \Sb^2 \setminus \{\Theta,-\Theta \}$, setting $z = \frac{\Theta + \sigma}{2}$, we remark that $0 < |z| < 1$, and that $z$ is orthogonal to $\Theta-z$, see Fig.~\ref{fig:zAtransform}.

On the other hand, for any $z \in \R^3$ such that $0 < |z| < 1$, the set of all $(\Theta,\sigma) \in (\Sb^2)^2$ such that $z = \frac{\Theta + \sigma}{2}$ is a (double) circle (represented on Fig.~\ref{fig:zAtransform}) of center $z$, radius $\sqrt{1-|z|^2}$, and inside in the plane $\{z \}^{\perp}$, which can be parametrized by $A \in \Sb^1$, see Fig.~\ref{fig:zAtransformcircle}. We denote by $z^{\perp_A}$ the projection on $\{z\}^{\perp}$ of the vector of this circle corresponding to the parameter $A \in \Sb^1$. Then $(z + z^{\perp_A},z - z^{\perp_A})$ defines a unique couple $(\Theta,\sigma) \in (\Sb^2)^2 \setminus \left( \left\{(\Theta,\Theta) \; | \; \Theta \in \Sb^2 \right\} \cup \left\{(\Theta,-\Theta) \; | \; \Theta \in \Sb^2 \right\} \right)$.

\begin{figure}[!ht]
    \centering
    
\begin{subfigure}[b]{0.58\textwidth}
\centering
    \def\r{1.5} \tdplotsetmaincoords{60}{125}
\begin{tikzpicture}[tdplot_main_coords,scale=2.4]
\begin{scope}[thin,black!30]
\draw[tdplot_screen_coords] (0,0,0) circle (\r);
\tdplotCsDrawLatCircle{\r}{0}
\draw (0,0,0) node[above right]{$O$};
\draw (0,0,0) node{$\ast$};
\draw (3,3.5,4) node{$\Sb^2$};
\end{scope}
\draw (0,.2,1.79) node[above]{$\Theta \equiv z + z^{\perp_A}$};
\draw (0,.2,1.79) node{$\ast$};
\draw (1,-1,-.4) node[below left]{$\sigma$};
\draw (1,-1,-.4) node{$\ast$};

\draw (0.5,-0.4,0.7) node[above left]{$\frac{\Theta + \sigma}{2} \equiv z$};
\draw (0.5,-0.4,0.7) node{$\ast$};

\draw[blue] (0.5,-0.4,0.7) -- (0,0,0);
\draw[red] (1,-1,-.4) -- (0,0,0);
\draw[red] (0,.2,1.79) -- (0,0,0);

\draw (.45,-.34,.81) -- (.4,-.3,.74);
\draw (.45,-.36,.63) -- (.4,-.3,.74);

\draw[blue] (.25,-.2,.35) node[above right]{$|z|$};
\draw[red] (.5,-.5,-.2) node[above]{$1$};

\draw[orange] (1,-1,-.4) -- (0.5,-0.4,0.7);
\draw[orange] (.75,-.7,.15) node[above left]{$\sqrt{1-|z|^2}$};

\draw (0.5,-0.4,0.7) --  (0,.2,1.79);

\draw (-.4,-.7,.6) node[left]{$z^{\perp_A}$};



\tdplotCsDrawCircle{\r}{-40}{40}{30}
\end{tikzpicture}
\caption{Representation in the 3D ball of radius $1$.}
\label{fig:zAtransform}
\end{subfigure}
\hfill
\begin{subfigure}[b]{0.41\textwidth}
\centering
\begin{tikzpicture}[scale=2.4]
\draw (0,0,0) node[above left]{$z$};
\draw (0,0,0) node{$\ast$};
\draw[opacity=.3] (0,0) circle (1);
\draw[opacity=.3] (.4,1.1) node{$\Sb^1$};
\draw (0,0) circle (.7);
\draw (.4,.58) node{$\ast$};
\draw (.4,.58) node[right]{$\Theta$};
\draw (-.4,-.58) node{$\ast$};
\draw (-.4,-.58) node[above right]{$\sigma$};
\draw (.568,.823) node{$\ast$};
\draw (.568,.823) node[above right]{$A$};
\draw (-.568,-.823) node{$\ast$};
\draw (-.568,-.823) node[below left]{$-A$};

\draw[dashed] (-.57,-.83) -- (.57,.83);

\draw[orange] (0,0) -- (.7,0);
\draw[orange] (.35,0) node[above]{$\sqrt{1 - |z|^2}$};

\end{tikzpicture}
\caption{Projection on plane $\{ z\}^{\perp}$ providing $A$.}
\label{fig:zAtransformcircle}
\end{subfigure}
\caption{Geometric configuration between $\sigma$, $\Theta$ and $z$.}

\end{figure}
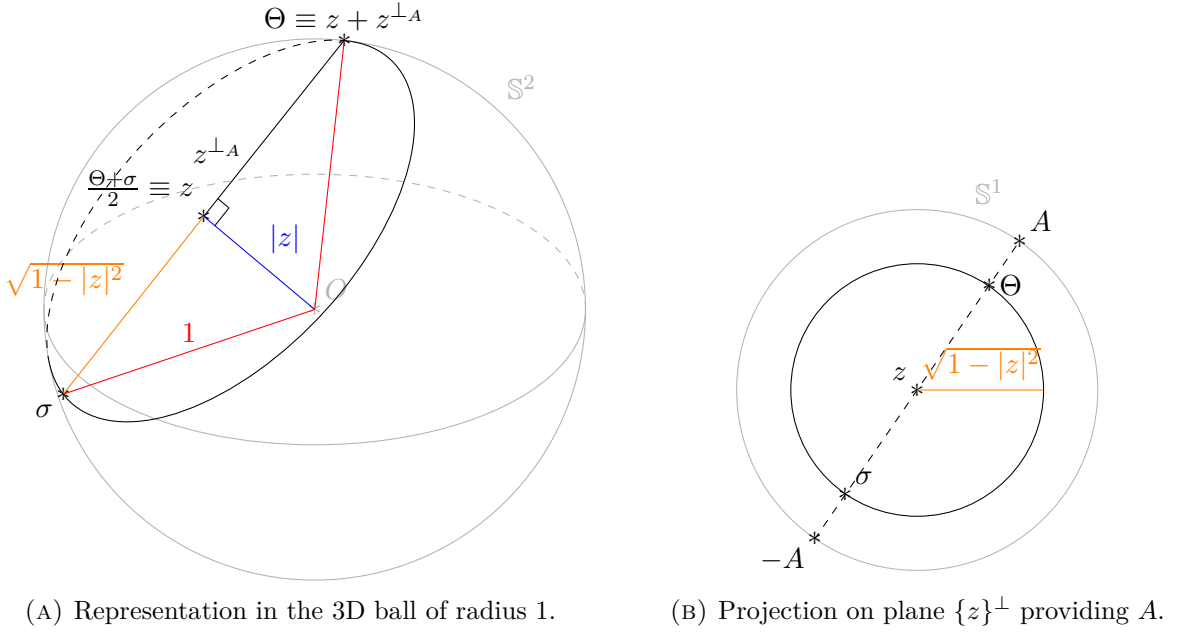

It is worth noting that, for all $z \in \R^3$ such that $0 < |z| < 1$, the parametrization of the associated circle represented on Fig.~\ref{fig:zAtransformcircle} by $A \in \Sb^1$ is defined up to the choice of an additive constant. Without being explicit about this choice, we simply state that there is a consistent choice of angle such that the transformation $(\Theta,\sigma) \mapsto (z,A)$ is a $\mathcal{C}^1$-diffeomorphism from $(\Sb^2)^2 \setminus \left( \left\{(\Theta,\Theta) \; | \; \Theta \in \Sb^2 \right\} \cup \left\{(\Theta,-\Theta) \; | \; \Theta \in \Sb^2 \right\} \right)$ to $\left\{z \in \R^3 \; | \; 0 < |z| < 1 \right\} \times \Sb^1$, that is almost everywhere from $(\Sb^2)^2$ to $\mathcal{B}(0,1) \times \Sb^1$, where $\mathcal{B}(0,1)$ stands for the open ball of $\R^3$ of center 0 and radius 1.

\begin{proposition} \label{prop:jacobian}
The Jacobian of the $(z,A)$ change of variables equals $4|z|^{-1}$.
\end{proposition}

\begin{proof}
The $(z,A)$ change of variables is geometrically natural, but tricky to write explicitly. Hence, we compute its Jacobian $\mathcal{J}$ by studying the image measure associated to it and not by direct computation. Since a rotation induces no deformation, we can state that $\mathcal{J}$ depends neither on $A$ nor on $\frac{z}{|z|}$. Thus $\mathcal{J}(z,A)$ can be written as $\Tilde{\mathcal{J}}(|z|)$. 

Define, for all $r \in (0,1)$, 
$$
S(r) = \iint_{(\Sb^2)^2} \mathbf{1}_{\left| \frac{\Theta + \sigma}{2} \right| \leq r} \, \dd \Theta \, \dd \sigma.
$$
Applying the $(z,A)$ change of variables, we have, on the one hand,
\begin{align*}
 S(r) = \int_{\mathcal{B}(0,1)} \int_{\Sb^1} \mathbf{1}_{\left| z \right| \leq r} \, \mathcal{J}(z,A) \, \dd A \, \dd z = \int_{\mathcal{B}(0,1)} \int_{\Sb^1} \mathbf{1}_{\left| z \right| \leq r} \, \Tilde{\mathcal{J}}(|z|) \, \dd A \, \dd z  = 2 \pi \times 4 \pi \int_0^{r} \, \Tilde{\mathcal{J}}(\rho) \, \rho^2 \, \dd \rho,
\end{align*}
where we passed in spherical coordinates to obtain the last equality. On the other hand, we can explicitly compute $S(r)$. Indeed, we can write
\begin{align*}
S(r) &= \iint_{(\Sb^2)^2} \mathbf{1}_{\left| \frac{\Theta + \sigma}{2} \right| \leq r} \, \dd \Theta \, \dd \sigma = \left| \left\{ (\Theta,\sigma) \in (\Sb^2)^2 \text{ s.t. } \left|   \frac{\Theta + \sigma}{2} \right| \leq r \right\}  \right| \\
&= \left| \left\{ (\Theta,\sigma) \in (\Sb^2)^2 \text{ s.t. }  |\Theta|^2 + |\sigma|^2 + 2 \Theta \cdot \sigma  \leq 4 r^2 \right\}  \right| = \left| \left\{ (\Theta,\sigma) \in (\Sb^2)^2 \text{ s.t. } \Theta \cdot \sigma  \leq 2 r^2 - 1 \right\}  \right| \\
&= 4 \pi \, \left| \left\{ \sigma \in \Sb^2 \text{ s.t. } \begin{pmatrix} 0 \\ 0 \\ 1 \end{pmatrix} \cdot \sigma  \leq 2 r^2 - 1 \right\}  \right| = 4 \pi \left( 4 \pi - 4 \pi (1-r^2) \right) \\
&= 16 \pi^2 r^2,
\end{align*}
where the penultimate equality comes from the fact that the area of the spherical cap 
$$
\left\{ \sigma \in \Sb^2 \text{ s.t. } \begin{pmatrix} 0 \\ 0 \\ 1 \end{pmatrix} \cdot \sigma  > 2 r^2 - 1 \right\},
$$
corresponding to the cap above the black circle on Fig. \ref{fig:caps}, equals $2 \pi \times (2 - 2 r^2)$.
\begin{figure}[!ht]
    \centering
    \def\r{1.5} \tdplotsetmaincoords{85}{125}
\begin{tikzpicture}[tdplot_main_coords,scale=1.7]
\begin{scope}[thin,black!30]
\draw[tdplot_screen_coords] (0,0,0) circle (\r);
\tdplotCsDrawLatCircle{\r}{0}
\draw (0,0,0) node[above right]{$O$};
\draw (0,0,0) node{$\ast$};
\end{scope}
\draw (0,0,1.5) node{$\ast$};
\draw[red] (0,0,0) -- (0,0,1);
\tdplotCsDrawLatCircle{\r}{40}

\draw (0,0,1) node{$\ast$};
\draw[blue] (0,0,1) -- (0,0,1.5);
\draw[red] (0,0,.5) node[left]{$2r^2 - 1$};
\draw[blue] (0,0,1.25) node[right]{$2 - 2r^2$};

\end{tikzpicture}

\caption{Representation of the considered spherical cap.}
\label{fig:caps}
\end{figure}
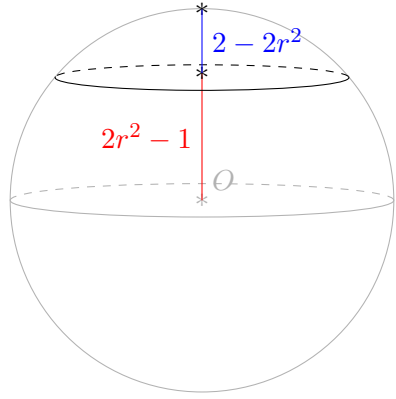

We conclude that, for all $r \in (0,1)$,
$$
\int_0^{r} \, \Tilde{\mathcal{J}}(\rho) \, \rho^2 \, \dd \rho = 2 r^2.
$$
Differentiating both sides with respect to $r$ provides the required value.
\end{proof}

\subsection{Study of the kernel form of $K^m_2$}
Recalling that, for almost every $v \in \R^3$,
\begin{multline*}
  K^m_2 h(v) = \int_{\R^3} \int_{\Sb^2} h \left(\frac{v+v_*}{2} + \frac{|v-v_*|}{2} \sigma \right) \,  M\left(\frac{v+v_*}{2} + \frac{|v-v_*|}{2} \sigma \right)^{-1/2} M(v)^{1/2} M(v_*) \\
\phantom{\int} B^m \left(|v-v_*|, |\cos(\widehat{v_*-v,\sigma})| \right) \, \dd \sigma \, \dd v_*,
\end{multline*}
we focus on the integration with respect to $v_*$, and perform the change of variable $V_* = v_* - v$ to obtain
\begin{multline*}
K^m_2 h(v) = \int_{\R^3} \int_{\Sb^2} h \left(v + \frac{V_*}{2} + \frac{|V_*|}{2} \sigma \right) \,  M\left(v + \frac{V_*}{2} + \frac{|V_*|}{2} \sigma \right)^{-1/2} M(v)^{1/2} M(V_* + v) \\
\phantom{\int} B^m(|V_*|,|\cos(\widehat{V_*,\sigma})|) \, \dd \sigma \, \dd V_*.
\end{multline*}
With the spherical coordinates, $V_* = \rho \Theta$,  $(\rho,\Theta) \in \R_+^* \times \Sb^2$, with Jacobian $\rho^2$, it becomes
\begin{multline*}
K^m_2 h(v) = \int_{\R_+^*} \iint_{(\Sb^2)^2} h \left(v + \rho \left(\frac{\Theta + \sigma}{2}\right)\right) \,  M\left(v + \rho \left(\frac{\Theta + \sigma}{2}\right)\right)^{-1/2} M(v)^{1/2} M(\rho \Theta + v) \\
\phantom{\int}  B^m(\rho,|\cos(\widehat{\Theta,\sigma})|) \, \dd \sigma \, \dd \Theta \, \rho^2 \, \dd \rho.
\end{multline*}
We now perform the $(z,A)$ change of variables detailed in Subsection \ref{subsectionzA} with Jacobian $4 |z|^{-1}$. Noticing that $z = \frac{\Theta + \sigma}{2}$, $\cos(\widehat{\Theta,\sigma}) = 2 |z|^2 - 1$ and $\Theta = z + z^{\perp_A}$, we thus obtain
\begin{multline}\label{eq:Km2afterchange}
K^m_2 h(v) = \int_{\R_+^*} \int_{\mathcal{B}(0,1)} \int_{\Sb^1} h \left(v + \rho z \right) \,  M\left(v + \rho z \right)^{-1/2} M(v)^{1/2} M(\rho \, (z + z^{\perp_A}) + v) \\
\phantom{\int} B^m(\rho, |2|z|^2 - 1|)\, \dd A \, 4|z|^{-1} \, \dd z \, \rho^2 \, \dd \rho. 
\end{multline}
Note that, in his proof \cite{grad1,grad2}, Grad wrote
$$
v' = v + p \quad  \text{ and } \quad v_* = v + p + q.
$$
By identification to our case, we have in fact
$$
p = \rho z \quad  \text{ and } \quad q  = \rho z^{\perp_A}.
$$

\begin{proposition} \label{proposition:k1kernel}
Setting, for almost every $v$, $p$, $\eta\in\R^3$, 
\begin{align}
\psi^m(v,p) &= 8 \pi \, c \int_{\R_+}  \exp\left( -\frac{r^2 + |v|^2 \, |\sin(\widehat{v,p})|^2}{2 k_B T_k} \right) \, \mathbb{I}_0 \left( \frac{  r \, |v| \, |\sin(\widehat{v,p})| }{ k_B T_k} \right) \label{eqdef:psim} \\ &\hspace{100pt}\times B^m \left( \sqrt{r^2 + |p|^2},\frac{|r^2 - |p|^2|}{r^2 + |p|^2} \right) \, \frac{r}{\sqrt{r^2 + |p|^2}}   \,  \ddd r, \nonumber \\
\kappa^m(v,\eta) &=  e^{ -\frac{ |\eta - v|^2}{8 k_B T_k}} \, e^{ -\frac{ 1}{8 k_B T_k} \, \frac{(|\eta|^2 - |v|^2)^2}{|\eta - v|^2} }  \, |\eta - v|^{-1} \, \psi^m(v,\eta - v),\phantom{\int}\label{eqdef:kappam}
\end{align}
where $\mathbb{I}_0$ stands for the modified Bessel function of the first kind of order $0$, defined in \eqref{eqdef:bessel}, then $K^m_2h$ can be written as
\begin{equation} \label{eq:kernelformofkm2}
K^m_2 h(v) = \int_{\R^3} \, h (\eta ) \, \kappa^m(v,\eta) \, \ddd \eta, \qquad \text{for a.e. } v \in \R^3.
\end{equation}
\end{proposition}

\begin{proof}
From \eqref{eq:Km2afterchange}, we immediately get
\begin{multline}\label{eq:Km2again}
K^m_2 h(v) = 4 \, c \int_{\R_+^*} \int_{\mathcal{B}(0,1)} \int_{\Sb^1} h \left(v + \rho z \right) \,  \exp\left(\frac{|v + \rho z|^2}{4 k_B T_k} - \frac{|v|^2}{4 k_B T_k} - \frac{|v + \rho z + \rho z^{\perp_A}|^2}{2 k_B T_k} \right)\\  \phantom{\int}B^m(\rho, |2|z|^2 - 1|)  \,|z|^{-1} \,  \rho^2 \, \dd A \, \dd z\, \dd \rho. 
\end{multline}
In the following, let us denote by $w$ the orthogonal projection of $v$ on $\{z\}^{\perp}$, {\it i.e.}
$$
w = v - \left(v \cdot \frac{z}{|z|} \right) \,  \frac{z}{|z|}.
$$
Then, following the idea of \cite{MR3005625}, we can successively write
\begin{align*}
    -\frac14 |v|^2 + \frac14|v+\rho z|^2  - \frac12|v+\rho z + \rho z^{\perp_A}|^2 = -\frac18|\rho z|^2 - \frac12 \left|\rho z^{\perp_A} + v + \frac12 \rho z \right|^2 \\
    = -\frac18\rho^2 |z|^2 - \frac12|\rho z^{\perp_A} + w|^2 - \frac12 \left|\left(v \cdot \frac{z}{|z|} \right) \,  \frac{z}{|z|} + \frac12 \rho z \right|^2 \\
    = -\frac18\rho^2 |z|^2 - \frac12|\rho z^{\perp_A} + w|^2 -\frac12 \left( \left(v + \frac12 \rho z \right) \cdot \frac{z}{|z|} \right)^2,
\end{align*}
since both $w$ and $z^{\perp_A}$ are orthogonal to $z$. Thus \eqref{eq:Km2again} becomes
\begin{multline}\label{eq:Km2againandagain}
K^m_2 h(v) = 4 \, c \int_{\R_+^*} \int_{\mathcal{B}(0,1)} h \left(v + \rho z \right) \,  e^{ -\frac{\rho^2 |z|^2}{8 k_B T_k}} \, e^{ -\frac{\left( \left(v + \frac12 \rho z \right) \cdot \frac{z}{|z|} \right)^2}{2 k_B T_k}}  \, |z|^{-1} \, \int_{\Sb^1} e^{ -\frac{|\rho z^{\perp_A} + w|^2}{2 k_B T_k}} \, \dd A \\ \phantom{\int} B^m(\rho, |2|z|^2 - 1|)  \, \dd z \, \rho^2 \, \dd \rho. 
\end{multline}
We now focus on the integral in the variable $A$. We have
\begin{multline*}
|\rho z^{\perp_A} + w|^2 = \rho^2 |z^{\perp_A}|^2 + |w|^2 + 2 \rho |z^{\perp_A}| \, |w| \, \cos(\widehat{z^{\perp_A},w}) \\
= \rho^2 \, (1 - |z|^2) + |w|^2 + 2 \rho \sqrt{1 - |z|^2} \, |w| \, \cos(\widehat{z^{\perp_A},w}).
\end{multline*}
Moreover, since $w$ is the projection of $v$ on the plane $\{z\}^{\perp}$, $|w| = |v| \, |\sin(\widehat{v,z})|$. Thus
$$
|\rho z^{\perp_A} + w|^2 = \rho^2 (1 -|z|^2) + |v|^2 \, \sin(\widehat{v,z})^2 + 2 \, \rho \, \sqrt{1-|z|^2 } \, |v| \, |\sin(\widehat{v,z})|  \; \cos(\widehat{z^{\perp_A},w}).
$$
For the sake of simplicity, let us momentarily denote
\begin{equation*}
X =  \frac{ \rho \, \sqrt{1-|z|^2 }  \, |v| \, |\sin(\widehat{v,z})|}{ k_B T_k},
\end{equation*}
which does not depend on $A$. Consequently, we can write
\begin{equation}
\int_{\Sb^1} \exp\left( -\frac{|\rho z^{\perp_A} + w|^2}{2 k_B T_k} \right) \, \dd A = \exp\left( -\frac{\rho^2 - |\rho z|^2 + |v|^2 \, |\sin(\widehat{v,z})|^2}{2 k_B T_k} \right) \int_{\Sb^1}  \exp \left( - X \cos(\widehat{z^{\perp_A},w}) \right) \, \dd A. \label{eq:intinA}
\end{equation}
Since the integration is performed on the whole unit circle $\Sb^1$, and because $w$ and $z^{\perp_A}$ belong to the same plane $\{ z\}^{\perp}$ (see Fig. \ref{fig:zAtransform}),
\begin{equation}
\int_{\Sb^1} \exp \left( - X \cos(\widehat{z^{\perp_A},w}) \right) \, \dd A =  \int_{0}^{2 \pi} e^{X \cos \theta } \, \dd \theta = 2 \pi \, \mathbb{I}_0(X).
\label{eq:intinAwithbessel}
\end{equation}
Consequently, thanks to \eqref{eq:intinA}--\eqref{eq:intinAwithbessel}, \eqref{eq:Km2againandagain} becomes
\begin{multline*}
K^m_2 h(v) = 8 \pi \, c \int_{\R_+^*} \int_{\mathcal{B}(0,1)} h \left(v + \rho z \right) \,  e^{ -\frac{\rho^2 |z|^2}{8 k_B T_k}} \, e^{ -\frac{\left( \left(v + \frac12 \rho z \right) \cdot \frac{z}{|z|} \right)^2}{2 k_B T_k}}  \,\exp\left( -\frac{\rho^2 - |\rho z|^2 + |v|^2 \, \sin(\widehat{v,z})^2}{2 k_B T_k}\right) \\
\mathbb{I}_0 \left( \frac{ \sqrt{\rho^2 - |\rho z|^2} \, |v| \, |\sin(\widehat{v,z})| }{ k_B T_k}  \right)  B^m(\rho, |2|z|^2 - 1|)  \, |z|^{-1} \, \rho^2 \, \dd z \,  \dd \rho.
\end{multline*}
We now make the change of variable $p = \rho z$, $\dd p = \rho^3 \, \dd z$, and use Fubini's theorem to get
\begin{multline} \label{eq:calculs}
K^m_2 h(v) 
 = 8 \pi \, c \int_{\R^3}  h \left(v + p \right) \,  e^{ -\frac{ |p|^2}{8 k_B T_k}} \, e^{ -\frac{\left( \left(v + \frac12 p \right) \cdot \frac{p}{|p|} \right)^2}{2 k_B T_k}}  \, |p|^{-1} \times \\ \left[\int_{|p|}^{\infty}  \exp\left( -\frac{\rho^2 - |p|^2 + |v|^2 \, \sin(\widehat{v,p})^2}{2 k_B T_k} \right) \mathbb{I}_0 \left( \frac{ \sqrt{\rho^2 - |p|^2} \, |v| \, |\sin(\widehat{v,p})| }{ k_B T_k}  \right)  B^m\left(\rho, \left|2\left|\frac{p}{\rho} \right|^2 - 1 \right| \right)  \, \dd \rho\right] \dd p. 
\end{multline}
When performing the change of variable $r = \sqrt{\rho^2 - |p|^2}$ of Jacobian $\frac{r}{\sqrt{r^2 + |p|^2}}$, the integral with respect to $\rho$ in \eqref{eq:calculs} becomes
\begin{align*}
\int_{\R_+}  \exp\left( -\frac{r^2 + |v|^2 \, \sin(\widehat{v,p})^2}{2 k_B T_k} \right) \, \mathbb{I}_0 \left( \frac{ r \, |v| \, |\sin(\widehat{v,p})| }{ k_B T_k}  \right)  B^m \left( \sqrt{r^2 + |p|^2},\frac{|r^2 - |p|^2|}{r^2 + |p|^2} \right) \, \frac{r}{\sqrt{r^2 + |p|^2}}  \, \dd r.
\end{align*}
Hence, using \eqref{eqdef:psim}, \eqref{eq:calculs} can be written as
$$
K^m_2 h(v) =  \int_{\R^3}  h \left(v + p \right) \,  e^{ -\frac{ |p|^2}{8 k_B T_k}} \, e^{ -\frac{\left( \left(v + \frac12 p \right) \cdot \frac{p}{|p|} \right)^2}{2 k_B T_k}}  \, |p|^{-1} \, \psi^m(v,p) \,  \dd p.
$$
Finally, noticing that, for all $v$, $\eta \in \R^3$,
$$
\left(v + \frac12 (\eta - v) \right) \cdot (\eta - v) = \frac12 (\eta + v) \cdot (\eta - v) = \frac12 \left(|\eta|^2 - |v|^2 \right),
$$
the change of variable $\eta = v + p$ allows to obtain \eqref{eq:kernelformofkm2}.
\end{proof}

\medskip

Now that we have obtained a kernel form of the operator $K^m_2$, we study its kernel $\kappa^m$, for which we intend to find upper bounds. We start by focusing on the function $\psi^m$.

\begin{lemma} \label{lemma:psiboundvarphiphi}
Let us denote, for all $\lambda \geq 0$ and $\alpha \in (- 1,1]$,
\begin{equation} \label{eqdef:varphialpha}
\varphi_{\alpha}(\lambda) = \int_{\R_+}  \exp\left( -\frac{(r-\lambda)^2}{2 k_B T_k} \right) \, \frac{r^{\alpha}}{1 + \sqrt{r \lambda}}  \, \ddd r,
\end{equation}
Then there exists $C > 0$ such that, for almost every $v$, $p \in \R^3$, we have, setting $\lambda = |v| \, |\sin(\widehat{v,p})|$,
\begin{equation} \label{eq:psimandvarphialpha}
\psi^m(v,p) \leq C \Big( (1 + |p|) \varphi_1(\lambda) + 1 + \varphi_{-\delta_1}(\lambda) + |p|^{-\delta_2} \varphi_{\delta_2}(\lambda) \Big).
\end{equation}
\end{lemma}

\medskip

\begin{proof}
Using the notation $\lambda = |v| \, |\sin(\widehat{v,p})|$, \eqref{eqdef:psim} becomes
\begin{equation} \label{eq:formofpsi0}
\psi^m(v,p) = 8 \pi \, c \int_{\R_+}  \exp\left( -\frac{r^2 + \lambda^2}{2 k_B T_k} \right) \, \mathbb{I}_0 \left( \frac{  r \, \lambda }{ k_B T_k} \right) \, B^m \left( \sqrt{r^2 + |p|^2},\frac{|r^2 - |p|^2|}{r^2 + |p|^2} \right) \, \frac{r}{\sqrt{r^2 + |p|^2}}   \,  \dd r.
\end{equation}
We know from \cite[Section 9.7.1, p. 377]{MR0167642}, that
\begin{equation} \label{eq:estimate1bessel}
\mathbb{I}_0(X) \underset{X \to + \infty}{\sim} \frac{e^{X}}{\sqrt{2 \pi X}}.
\end{equation}
It is moreover obvious that, for any $X \geq 0$,
\begin{equation} \label{eq:estimate2bessel}
\mathbb{I}_0(X) = \frac{1}{2 \pi} \int_{0}^{2 \pi} e^{ \, X \cos \theta } \, \dd \theta \leq e^X.
\end{equation}
Combining estimates \eqref{eq:estimate1bessel}--\eqref{eq:estimate2bessel}, we conclude that there exists $C > 0$ such that, for any $X \geq 0$,
\begin{equation} \label{eq:estimate3bessel}
\mathbb{I}_0(X) \leq C \, \frac{e^{X}}{1 + \sqrt{X}}.
\end{equation}
We then inject estimate \eqref{eq:estimate3bessel} into \eqref{eq:formofpsi0} to obtain
\begin{align*}
\psi^m(v,p) &\leq C \int_{\R_+} \exp \left( -\frac{r^2 + \lambda^2 - 2 r \lambda}{2 k_B T_k} \right) \frac{1}{1 + \sqrt{r \lambda}} \, B^m \left( \sqrt{r^2 + |p|^2},\frac{|r^2 - |p|^2|}{r^2 + |p|^2} \right) \, \frac{r}{\sqrt{r^2 + |p|^2}}   \,  \dd r \\
&\leq C \int_{\R_+} \exp \left( -\frac{(r-\lambda)^2}{2 k_B T_k} \right) \frac{1}{1 + \sqrt{r \lambda}} \, B^m \left( \sqrt{r^2 + |p|^2},\frac{|r^2 - |p|^2|}{r^2 + |p|^2} \right) \, \frac{r}{\sqrt{r^2 + |p|^2}}   \,  \dd r.
\end{align*}
Now, note that, if $|\cos \theta | = \frac{|r^2 - |p|^2|}{r^2 + |p|^2}$, then
$$
|\sin \theta | = \sqrt{1 -  \left(\frac{|r^2 - |p|^2|}{r^2 + |p|^2}\right)^2} = \sqrt{\frac{4 r^2 |p|^2}{\left(r^2 + |p|^2\right)^2}} = \frac{2 r |p|}{r^2 + |p|^2}.
$$
Thus it comes from the assumption \eqref{eq:bmmajor} on the cross section that
\begin{equation} \label{eq:kernelboundrp}
B^m \left( \sqrt{r^2 + |p|^2},\frac{|r^2 - |p|^2|}{r^2 + |p|^2} \right) \, \frac{r}{\sqrt{r^2 + |p|^2}} \leq C \left( r |p| + \frac{r |p|}{(r^2 + |p|)^{3/2}} + r + r^{-\delta_1} + r^{\delta_2} \ |p|^{-\delta_2} \right),
\end{equation}
because $0 \leq \delta_2 < 1$, $0 \leq \delta_2 < \frac12$, and using the fact that $r \leq \sqrt{r^2 + |p|^2}$. Injecting \eqref{eq:kernelboundrp} in \eqref{eq:formofpsi0} ensures that
\begin{multline} \label{eq:boundmultipsim}
\psi^m(v,p) \leq C(1+|p|) \int_{\R_+} \exp \left( -\frac{(r-\lambda)^2}{2 k_B T_k} \right) \frac{r}{1 + \sqrt{r \lambda}}  \,  \dd r + C \int_{\R_+} \exp \left( -\frac{(r-\lambda)^2}{2 k_B T_k} \right) \frac{r |p|}{(r^2 + |p|)^{3/2}}  \,  \dd r \\
+ C \int_{\R_+} \exp \left( -\frac{(r-\lambda)^2}{2 k_B T_k} \right) \frac{r^{-\delta_1}}{1 + \sqrt{r \lambda}}  \,  \dd r + \frac C{|p|^{\delta_2}} \int_{\R_+} \exp \left( -\frac{(r-\lambda)^2}{2 k_B T_k} \right) \frac{r^{\delta_2}}{1 + \sqrt{r \lambda}}  \,  \dd r.
\end{multline}
Finally, we notice that, for any $\lambda \geq 0$ and $p \in \R^3 \setminus \{ 0\}$,
\begin{align*}
&\int_{\R_+}  \exp\left( -\frac{(r-\lambda)^2}{2 k_B T_k} \right) \, \frac{r |p|}{(r^2 + |p|^2)^{3/2}}  \, \dd r \leq  |p| \int_{\R_+} \frac{r}{(r^2 + |p|^2)^{3/2}}  \, \dd r = |p| \left[-(r^2 + |p|^2)^{-1/2}\right]_{0}^{\infty} = 1.
\end{align*}
Plugging that previous estimate in \eqref{eq:boundmultipsim} implies \eqref{eq:psimandvarphialpha}.
\end{proof}

\smallskip

\begin{lemma} \label{lemma:varphialpha}
There exists $C>0$ such that, for almost every $v$, $p \in \R^3$
\begin{equation} \label{eq:finalboundpsim}
\psi^m(v,p) \leq C \, \left(|p| + |p|^{-\delta_2} \right).
\end{equation}
\end{lemma}

\begin{proof}
We first focus on the functions $\varphi_{\alpha}$ defined by \eqref{eqdef:varphialpha} and prove that for any $\alpha \in (-1,1]$, there exists $C_{\alpha}>0$ such that, for any $\lambda \geq 0$,
\begin{equation} \label{eq:varphialphabound}
\varphi_{\alpha}(\lambda) \leq C_{\alpha}.
\end{equation}
Let us first consider the case when $\alpha \in (-1,0]$. Then, for any $\lambda \geq 0$,
\begin{equation} \label{eq:varphialphacase-10}
    \begin{split}
        \varphi_{\alpha}(\lambda) &\leq \int_{\R_+} r^{\alpha} \,  \exp \left( -\frac{(r-\lambda)^2}{2 k_B T_k} \right)  \dd r \leq \int_{0}^1 r^{\alpha} \,  \dd r + \int_1^{+\infty} \exp \left( -\frac{(r-\lambda)^2}{2 k_B T_k} \right)  \dd r \\ &\leq \frac{1}{1+\alpha} + \int_{\R} \exp \left( -\frac{r^2}{2 k_B T_k} \right)  \dd r,
    \end{split}
\end{equation}
which does not depend on $\lambda$. We now focus on the case $0 < \alpha \leq 1$. First, note that
\begin{equation} \label{eq:varphialphacase01}
    \begin{split}
&\int_{\R_+} r^{\alpha} \,  \exp \left( -\frac{(r-\lambda)^2}{2 k_B T_k} \right)  \dd r \leq \int_{-\lambda}^{\infty} (r + \lambda)^{\alpha} \, \exp \left( -\frac{r^2}{2 k_B T_k} \right)  \dd r \\
&\leq \int_{\R} |r|^{\alpha}  \exp \left( -\frac{r^2}{2 k_B T_k} \right)  \dd r  + \, \lambda^{\alpha} \, \int_{\R}  \exp \left( -\frac{r^2}{2 k_B T_k} \right)  \dd r  \leq C_{\alpha} \, (1 + \lambda^{\alpha}).
    \end{split}
\end{equation}
Consequently, we deduce, for any $\lambda \geq 0$, that
\begin{equation} \label{eq:estimationvarphialpha}
\varphi_{\alpha}(\lambda) \leq \int_{\R_+} r^{\alpha} \,  \exp \left( -\frac{(r-\lambda)^2}{2 k_B T_k} \right)  \dd r \leq C_{\alpha} \, (1 + \lambda^{\alpha}).
\end{equation}
Besides, we have, for any $\lambda > 0$,
\begin{equation} \label{eq:estimationvarphialpha2}
\varphi_{\alpha}(\lambda) \leq \frac{1}{\sqrt{\lambda}} \int_{\R_+} r^{\alpha-\frac12} \,  \exp \left( -\frac{(r-\lambda)^2}{2 k_B T_k} \right)  \dd r.
\end{equation}
When $\alpha \in (0,\frac12]$, applying \eqref{eq:varphialphacase-10} for $\alpha - \frac12 \in (-\frac12,0]$ in \eqref{eq:estimationvarphialpha2} implies that, for any $\lambda > 0$,
\begin{equation} \label{eq:estimationvarphialpha12}
\varphi_{\alpha}(\lambda) \leq \frac{C_{\alpha}}{\sqrt{\lambda}}.
\end{equation}
Combining \eqref{eq:estimationvarphialpha} and \eqref{eq:estimationvarphialpha12} thus yields, when $0 < \alpha \leq \frac12$, that, for any $\lambda \geq 0$,
$$
\varphi_{\alpha}(\lambda) \leq C_{\alpha}.
$$
When $\frac12 < \alpha \leq 1$, applying \eqref{eq:varphialphacase01} in \eqref{eq:estimationvarphialpha2} implies that, for any $\lambda \geq 0$,
$$
\int_{\R_+} r^{\alpha-\frac12} \,  \exp \left( -\frac{(r-\lambda)^2}{2 k_B T_k} \right)  \dd r \leq C_{\alpha} \, \left(1 + \lambda^{\alpha - \frac12} \right),
$$
so that
\begin{equation} \label{eq:estimationvarphialphaautre}
\varphi_{\alpha}(\lambda) \leq C_{\alpha} \, \left( \frac{1}{\sqrt{\lambda}} + \lambda^{\alpha - 1} \right).
\end{equation}
Combining \eqref{eq:estimationvarphialpha} and \eqref{eq:estimationvarphialphaautre} allows to obtain \eqref{eq:varphialphabound} when $\frac12 < \alpha \leq 1$.

\noindent Eventually, using jointly Lemma \ref{lemma:psiboundvarphiphi}, and \eqref{eq:varphialphabound} with $\alpha \in \{1,-\delta_1,\delta_2\} \subset (-1,1]$, we obtain \eqref{eq:finalboundpsim}.
\end{proof}

\smallskip
\noindent The bound \eqref{eq:finalboundpsim} obtained on $\psi^m$ allows us to obtain both following propositions, implying the compactness of $K^m_2$, and also used in the main body of this paper for the proof of the compactness of the operator $K_2$.

\medskip

\begin{proposition} \label{prop:kappa0bound}
There exists $C>0$ such that
\begin{align}
&\int_{\R^3} \kappa^m(v,\eta) \, \ddd \eta \leq \frac{C}{1 + |v|},  \qquad \text{for a.e. } v\in \R^3,  \label{eq:kappakbound1} \\
  &\int_{\R^3} \kappa^m(v,\eta) \, \ddd v \leq C, \qquad  \qquad \;  \text{for a.e. } \eta \in \R^3. \label{eq:kappakbound2}
\end{align}
\end{proposition}

\begin{proof}
By definition \eqref{eqdef:kappam} of $\kappa^m$ and from \eqref{eq:finalboundpsim}, we have
\begin{equation} \label{eq:firstboundkappam}
\kappa^m (v,\eta)   \leq C  e^{ -\frac{ |\eta-v|^2}{8 k_B T_k}} \, e^{ -\frac{ 1}{8 k_B T_k} \, \frac{(|\eta|^2 - |v|^2)^2}{|\eta - v|^2} }  \, (1 + |\eta-v|^{-1 -\delta_2}).
\end{equation}
Integrating \eqref{eq:firstboundkappam} with respect to $\eta$ and making the change of variable $\eta \mapsto \eta - v = p$ yield
$$
\int_{\R^3} \kappa^m (v,\eta)  \, \dd \eta  \leq C  \int_{\R^3} e^{ -\frac{ |p|^2}{8 k_B T_k}} \, e^{ - \frac{\left(|p| + 2 \, |v| \, \cos(\widehat{v,p})  \right)^2}{8 k_B T_k} } \, (1 + |p|^{-1 -\delta_2})   \,  \dd p.
$$
Passing to spherical coordinates, the integral in $p$ on the right-hand side writes
\begin{align*}
&\int_{\R_+} \int_{0}^{2 \pi} \int_0^{\pi} (\rho^{2} +\rho^{1-\delta_2}) \, e^{ -\frac{ \rho^{2}}{8 k_B T_k}} \, e^{ - \frac{\left(\rho + 2 \, |v| \, \cos  \theta   \right)^2}{8 k_B T_k} }  \, \sin \theta  \, \dd \theta \, \dd \psi \, \dd \rho \\
&= 2 \pi \int_{\R_+} (\rho^{2} +\rho^{1-\delta_2}) \, e^{ -\frac{ \rho^2}{8 k_B T_k}} \left( \int_0^{\pi} e^{ - \frac{\left(\rho + 2 \, |v| \, \cos \theta   \right)^2}{8 k_B T_k} }  \,  \sin \theta  \, \dd \theta  \right) \dd \rho 
\end{align*}
Performing the change of variable $s = \cos \theta $, we obtain
\begin{equation} \label{eq:estimate1proofkappaalpha}
\int_0^{\pi} e^{ - \frac{\left(\rho + 2 \, |v| \, \cos \theta   \right)^2}{8 k_B T_k} }  \,  \sin \theta  \, \dd \theta \leq \int_{-1}^{1} e^{ - \frac{\left(\rho + 2 \, |v| \, s  \right)^2}{8 k_B T_k} }  \, \dd s \leq 2.
\end{equation}
Besides, performing the change of variable $r = \rho + 2 \, |v| \, s$, we obtain
\begin{equation} \label{eq:estimate2proofkappaalpha}
\int_{-1}^{1} e^{ - \frac{\left(\rho + 2 \, |v| \, s  \right)^2}{8 k_B T_k} }  \, \dd s = \frac{1}{2 |v|} \,https://www.overleaf.com/project/6228932d38e1d843c1f92a9e \int_https://www.overleaf.com/project/6228932d38e1d843c1f92a9e{\rho - 2 |v|}^{\rho + 2 |v|} e^{ - \frac{r^2}{8 k_B T_k} }  \, \dd r \leq \frac{1}{2 |v|} \, \int_{\R} e^{ - \frac{r^2}{8 k_B T_k} }  \, \dd r = \frac{C}{|v|}.
\end{equation}
Combining estimates \eqref{eq:estimate1proofkappaalpha}--\eqref{eq:estimate2proofkappaalpha} together yields
$$
\int_0^{\pi} e^{ - \frac{\left(\rho + 2 \, |v| \, \cos \theta   \right)^2}{8 k_B T_k} }  \,  \sin \theta  \, \dd \theta  \leq \frac{C}{1 + |v|}.
$$
Hence
$$
\int_{\R^3} \kappa^m (v,\eta)  \, \dd \eta \leq \frac{C }{1 + |v|} \int_{\R_+} (\rho^{2} +\rho^{1-\delta_2}) \, e^{ -\frac{ \rho^2}{8 k_B T_k}}.
$$
Since $1 - \delta_2 \geq 0$, we have $\displaystyle \int_{\R_+} (\rho^{2} +\rho^{1-\delta_2}) \, e^{ -\frac{ \rho^2}{8 k_B T_k}}\, \dd \rho < \infty$, and we can conclude for \eqref{eq:kappakbound1}. On the other hand, integrating \eqref{eq:firstboundkappam} with respect to $v$ and making the change of variable $v \mapsto v- \eta = p$ yields
\begin{align*}
\int_{\R^3} \kappa^m (v,\eta)  \, \dd v  &\leq C  \int_{\R^3}  e^{ -\frac{ |p|^2}{8 k_B T_k}} \, e^{ - \frac{\left(|p| + 2 \, |p+\eta| \, \cos(\widehat{p+\eta,p})  \right)^2}{8 k_B T_k} } \, (1 + |p|^{-1 -\delta_2})   \,  \dd p \\
&\leq C  \int_{\R^3}  e^{ -\frac{ |p|^2}{8 k_B T_k}} \, (1 + |p|^{-1 -\delta_2})   \,  \dd p.
\end{align*}
Since $-1 - \delta_2 > - 3$, we have $\displaystyle \int_{\R^3}  e^{ -\frac{ |p|^2}{8 k_B T_k}} \, (1 + |p|^{-1 -\delta_2})   \,  \dd p < \infty$, which allows to conclude for \eqref{eq:kappakbound2}.
\end{proof}

\smallskip

\begin{proposition}  \label{prop:kappa0L2loc}
The kernel $\kappa^m$ belongs to $L^2_{loc} \left(\R^3,  \ddd v \, ; \, L^2 \left(\R^3, \, \ddd \eta \right)\right)$.
\end{proposition}

The proof is similar to the one of Proposition \ref{prop:kappa0bound}, also remarking that $-2-2\delta_2 > -3$.

\section{Auxiliary lemmas}

\begin{lemma} \label{lemma:expdeltafrac}
For all $s > 0$, there exists $C > 0$ such that, for all $I,J \geq 0$,
$$
e^{-s |J-I|} \, \frac{1+I}{1+J} \leq C \quad \text{and} \quad e^{-s |J-I|}  \, \frac{1+J}{1+I} \leq C.
$$
\end{lemma}

\begin{proof}
First, we straightforwardly have, for $I \leq J$,
$$
e^{-s |J-I|} \, \frac{1+I}{1+J}  \leq 1.
$$
On the other hand, we focus on the case $I \geq J$. We can write
\begin{align*}
\text{log} \left( e^{-s |J-I|} \, \frac{1+I}{1+J} \right) &= - s (I-J) + \text{log}(1 + I) - \text{log}(1 + J) \\
&= s (1 + J) - \text{log}(1 + J) - \left(  s (1 + I) - \text{log}(1 + I) \right).
\end{align*}
Since $x \mapsto s x - \text{log}(x)$ is increasing on $\left[ \frac{1}{s}, \infty \right)$, we have, for all $I,J \geq (\frac{1}{s} - 1)_+$ such that $I \geq J$,
$$
e^{-s |J-I|} \, \frac{1+I}{1+J} \leq 1.
$$
When $(I,J) \in [0,(\frac{1}{s} - 1)_+]^2 \text{ with } I \geq J$, we have
$$
e^{-s |J-I|} \, \frac{1+I}{1+J} \leq \text{max} \left(1,\frac{1}{s} \right).
$$
When $ 0 \leq J \leq (\frac{1}{s}-1)_+ \leq I$, then
$$
e^{-s |J-I|} \, \frac{1+I}{1+J} = e^{-s I} (1 + I) \times \frac{e^{s J}}{1 + J} \leq e^{-s I}(1+I) \times e^{(1 - s)_+} \leq  \text{max} \left(1,\frac{1}{s} \right) \times e^{(1 - s)_+}.
$$
All in all, we obtain the first estimate. Simply exchanging the notation, the second one immediately comes from the first one.
\end{proof}

The following corollary is then a straightforward consequence of Lemma~\ref{lemma:expdeltafrac}.

\begin{corollary} \label{corollary:expdeltafrac}
For all $r \in \R$ and $s > 0$, there exists $C>0$ such that, for all $I,J \geq 0$,
\begin{equation} \label{eq:expdeltafrac}
e^{-s |J-I|} \, (1 + I)^r \leq C \, (1 + J)^r.
\end{equation}
\end{corollary}

\bibliographystyle{abbrv}
\bibliography{biblio}

\end{document}